\documentclass[sn-mathphys]{sn-jnl}
\jyear{2021}%
\usepackage{array}
\newcolumntype{P}[1]{>{\centering\arraybackslash}p{#1}}
\usepackage{tabularray}
\usepackage{booktabs}
\usepackage{amsmath,amssymb,amsthm,mathrsfs,amsxtra,bm}
\usepackage{verbatim,stmaryrd}     
\usepackage{graphics,graphicx,epsfig,subfig}
\usepackage{hyperref} 
\usepackage{float}
\usepackage{xspace}
\usepackage{enumitem}%
\usepackage{soul}
\usepackage{kantlipsum}

\usepackage{wrapfig}
\usepackage{physics,tabularx}
\theoremstyle{thmstyleone}%
\newtheorem{theorem}{Theorem}
\newtheorem{proposition}[theorem]{Proposition}%
\newtheorem{corollary}[theorem]{Corollary}

\theoremstyle{thmstyletwo}%
\newtheorem{remark}{Remark}%

\theoremstyle{thmstylethree}%
\newtheorem{definition}{Definition}%

\begin{document}

\title[Article Title]{A Note on the Hausdorff Distance between Norm Balls and their Linear Maps}

\author*[1]{\fnm{Shadi} \sur{Haddad}}\email{shhaddad@ucsc.edu}

\author[1]{\fnm{Abhishek} \sur{Halder}}\email{ahalder@ucsc.edu}

\equalcont{These authors contributed equally to this work.}

\affil*[1]{\orgdiv{Applied Mathematics}, \orgname{University of California at Santa Cruz}, \orgaddress{\street{1156 High Street}, \city{Santa Cruz}, \postcode{95064}, \state{California}, \country{USA}}}


\abstract{We consider the problem of computing the (two-sided) Hausdorff distance between the unit $\ell_{p_{1}}$ and $\ell_{p_{2}}$ norm balls in finite dimensional Euclidean space for $1 \leq p_1 < p_2 \leq \infty$, and derive a closed-form formula for the same. We also derive a closed-form formula for the Hausdorff distance between the $k_1$ and $k_2$ unit $D$-norm balls, which are certain polyhedral norm balls in $d$ dimensions for $1 \leq k_1 < k_2 \leq d$. When two different $\ell_p$ norm balls are transformed via a common linear map, we obtain several estimates for the Hausdorff distance between the resulting convex sets. These estimates upper bound the Hausdorff distance or its expectation, depending on whether the linear map is arbitrary or random. We then generalize the developments for the Hausdorff distance between two set-valued integrals obtained by applying a parametric family of linear maps to different $\ell_p$ unit norm balls, and then taking the Minkowski sums of the resulting sets in a limiting sense. To illustrate an application, we show that the problem of computing the Hausdorff distance between the reach sets of a linear dynamical system with different unit norm ball-valued input uncertainties, reduces to this set-valued integral setting.}


\maketitle

\section{Introduction}\label{sec:Intro}
Given compact $\mathcal{X},\mathcal{Y}\subset\mathbb{R}^{d}$, the two sided Hausdorff distance $\delta$ between them is a mapping $\delta:\mathcal{X}\times\mathcal{Y}\mapsto \mathbb{R}_{\geq 0}$ defined as
\begin{align}
\delta\left(\mathcal{X},\mathcal{Y}\right) := \max\bigg\{\underset{\bm{x}\in\mathcal{X}}{\sup}\:\underset{\bm{y}\in\mathcal{Y}}{\vphantom{\sup}\inf}\|\bm{x}-\bm{y}\|_{2},\:\underset{\bm{y}\in\mathcal{Y}}{\sup}\:\underset{\bm{x}\in\mathcal{X}}{\vphantom{\sup}\inf}\|\bm{x}-\bm{y}\|_{2}\bigg\},
\label{DefHausdorffDist}	
\end{align}
where $\|\cdot\|_2$ is the Euclidean norm with the associated scalar product $\langle\cdot,\cdot\rangle$. Denoting the unit 2 norm ball in $\mathbb{R}^{d}$ as $\mathbb{B}^{d}_2$, an equivalent definition of the Hausdorff distance is
\begin{align}
\delta\left(\mathcal{X},\mathcal{Y}\right) := \inf\{\lambda\geq 0 \mid \mathcal{X} \subset \mathcal{Y} \dotplus \lambda\mathbb{B}^{d}_2, \: \mathcal{Y} \subset \mathcal{X} \dotplus \lambda\mathbb{B}^{d}_2\},
\label{EquivDefHausdorffDist}    
\end{align}
where $\dotplus$ denotes the Minkowski sum. As is well-known \cite[p. 60-61]{schneider2014convex}, $\delta\geq 0$ is a metric. The distance was introduced by Hausdorff in 1914 \cite[p. 293ff]{hausdorff1914grundzuge}, and can be considered more generally on the set of nonempty closed and bounded subsets of a metric space $(\mathcal{M},{\rm{dist}})$ by replacing the Euclidean distance $\|\cdot\|_{2}$ in \eqref{DefHausdorffDist} with ${\rm{dist}}(\cdot,\cdot)$. The Hausdorff distance and the associated topology, have found widespread applications in mathematical economics \cite{hildenbrand2015core}, stochastic geometry \cite{stoyan2013stochastic}, set-valued analysis \cite{serra1998hausdorff}, image processing \cite{huttenlocher1993comparing} and pattern recognition \cite{jesorsky2001robust}. The distance $\delta$ has several useful properties with respect to set operations, see e.g., \cite[Lemma 2.2]{de1976differentiability}, \cite[Lemma A2]{serry2021overapproximating}.

The support function $h_{\mathcal{K}}(\cdot)$ of a compact convex set $\mathcal{K} \subset \mathbb{R}^{d}$, is given by
\begin{align}
h_{\mathcal{K}}(\bm{y}) := \underset{\bm{x}\in\mathcal{K}}{\sup}\:\{\langle\bm{y},\bm{x}\rangle \mid \bm{y}\in\mathbb{S}^{d-1}\},
\label{DefSptFn}	
\end{align}
where $\langle\cdot,\cdot\rangle$ denotes the standard Euclidean inner product, and $\mathbb{S}^{d-1}$ is the unit sphere in $\mathbb{R}^{d}$. The definition \eqref{DefSptFn} can be extended to any closed convex set $\mathcal{K}$ in the sense $h_{\mathcal{K}}=+\infty$ if and only if $\mathcal{K}$ is unbounded \cite[Prop. 2.1.3]{hiriart2013convex}. Geometrically, $h_{\mathcal{K}}(\bm{y})$ gives the signed distance of the supporting hyperplane of $\mathcal{K}$ with outer normal vector $\bm{y}$, measured from the origin. The support function $h_{\mathcal{K}}(\bm{y})$ uniquely determines the set $\mathcal{K}$. Since only the direction of the normal vector $\bm{y}$ matters, we restrict the domain of the support function on $\mathbb{S}^{d-1}$ instead of $\mathbb{R}^{d}$. Doing so, invites no loss generality because a support function $h_{\mathcal{K}}(\cdot)$ is always positive homogeneous of degree one (see e.g., \cite[p. 209]{hiriart2013convex}). Furthermore, $h_{\mathcal{K}}(\bm{y})$ is a convex function of $\bm{y}$. From \eqref{DefSptFn}, we note that for given $\bm{T}\in\mathbb{R}^{d^{\prime}\times d}$ and compact convex $\mathcal{K}\subset\mathbb{R}^{d}$, the support function of the compact convex set $\bm{T}\mathcal{K}\subset \mathbb{R}^{d^{\prime}}$ is
\begin{align}
h_{\mathcal{\bm{T}K}}(\bm{y}) = h_{\mathcal{K}}(\bm{T}^{\top}\bm{y}), \quad \bm{y}\in\mathbb{S}^{d^{\prime}-1}.
\label{SptFnLinearMap}    
\end{align}
For more details on the support function, we refer the readers to \cite[Ch. V]{hiriart2013convex}. 

The two-sided Hausdorff distance \eqref{DefHausdorffDist} between a pair of convex compact sets $\mathcal{K}_{1}$ and $\mathcal{K}_{2}$ in $\mathbb{R}^{d}$ can be expressed in terms of their respective support functions $h_{1}(\cdot), h_{2}(\cdot)$ as 
\begin{align}
\delta\left(\mathcal{K}_{1},\mathcal{K}_{2}\right) = \underset{\bm{y}\in \mathbb{S}^{d-1}}{\sup}\quad \big\vert h_{1}(\bm{y})- h_{2}(\bm{y})\big\vert,
\label{HausdorffSptFn} 
\end{align}
where the absolute value in the objective can be dispensed if one set is included in another\footnote{This is because $\mathcal{K}_{1}\subseteq\mathcal{K}_{2}$ if and only if $h_{1}(\bm{y}) \leq h_{2}(\bm{y})$ for all $\bm{y}\in\mathbb{S}^{d-1}$.}. Thus, computing $\delta$ leads to an optimization problem over all unit vectors $\bm{y}\in\mathbb{S}^{d-1}$.

The support function, by definition, is positive homogeneous of degree one. Therefore, the unit sphere constraint $\|\bm{y}\|_{2} = 1$ in \eqref{HausdorffSptFn} admits a lossless relaxation to the unit ball constraint $\|\bm{y}\|_{2} \leq 1$. Even so, problem \eqref{HausdorffSptFn} is nonconvex because its objective is nonconvex in general.

In this study, we consider computing \eqref{HausdorffSptFn} for the case when the sets $\mathcal{K}_1,\mathcal{K}_2$ are different unit norm balls and more generally, linear maps of such norm balls in an Euclidean space. This can be viewed as quantifying the conservatism in approximating a norm ball by another in terms of the Hausdorff distance. We show that computing the associated Hausdorff distances lead to optimizing the difference between norms over the unit sphere or ellipsoid. While bounds on the difference of norms over the unit cube have been studied  before \cite{shisha1967differences}, the optimization problems arising here seem new, and the techniques in \cite{shisha1967differences} do not apply in our setting. 

\textbf{Motivating application:} A practical motivation for our study comes from control theory and formal verification literature \cite{pecsvaradi1971reachable,witsenhausen1972remark,chutinan1999verification,kurzhanski1997ellipsoidal,varaiya2000reach,le2010reachability,althoff2021set,haddad2023curious,haddad2021anytime}. There, it is of interest to investigate how controlled dynamical systems evolve relative to each other subject to different set-valued input uncertainties. For example, if the controlled dynamical systems model vehicles driving on road, then one practical question is whether the set of states reachable by one vehicle at a specific time, can intersect the other set, possibly resulting in a collision. The different set-valued inputs in the vehicle context, represent respective actuation uncertainties. Then, a natural way to quantify safety or the lack of it, is by computing the distance between such sets in terms of the Hausdorff metric.

In such applications, for computational ease, one often assumes box-valued (i.e., $\ell_{\infty}$ norm ball) input uncertainty sets even though the true input uncertainty sets might be $\ell_{p}$ norm balls for $0<p<\infty$. Such computational approximation in the input uncertainty sets lead to over-approximation of the reach sets \cite[Sec. III]{haddad2023curious}. Then, quantifying the conservatism in over-approximation amounts to computing the Hausdorff distance between such reach sets. When the controlled dynamical systems are linear, it turns out that the corresponding Hausdorff distance \eqref{HausdorffSptFn} takes the form
$$\underset{\|\bm{y}\|_2=1}{\sup}\left(\displaystyle\int_{0}^{t}\|\bm{T}(\tau)\bm{y}\|_{q_2}-\|\bm{T}(\tau)\bm{y}\|_{q_1}\right)\:\differential\tau, \qquad 1\leq q_2 < q_1 \leq \infty,$$
which is what we investigate in Sec. \ref{sec:integral} in this paper. 

We also provide an application Example in Sec. \ref{sec:integral}, where the different reach sets result from the motion of a satellite subject to $\ell_2$ and $\ell_{\infty}$ norm ball-valued uncertain input sets. In this application, the input components denote the radial and tangential thrusts, and depending on the actuators installed (e.g., gas jets, reaction wheel), two different scenarios may arise: one where there are hard bounds on the magnitude of the thrust components (i.e., $\ell_{\infty}$ norm ball), and another in which there is bounded thrust magnitude (i.e., $\ell_2$ norm ball). So from an engineering perspective, it is natural to quantify the Hausdorff distance between the reach sets resulting from two different types of actuation uncertainties.  

\textbf{Related works:} There have been several works on designing approximation algorithms for computing the Hausdorff distance between convex polygons \cite{atallah1983linear}, curves \cite{belogay1997calculating}, images \cite{huttenlocher1993comparing}, meshes \cite{aspert2002mesh} or point cloud data \cite{taha2015efficient}; see also \cite{goffin1983relationship,alt1995approximate,alt2003computing,konig2014computational,jungeblut2021complexity}. There are relatively few \cite{marovsevic2018hausdorff} known exact formula for the Hausdorff distance between sets. To the best of the authors' knowledge, analysis of the Hausdorff distance between norm balls and their linear maps as pursued here, did not appear in prior literature.

\textbf{Contributions:} Our specific contributions are as follows.
\begin{list}{$\bullet$}{} 
    \item We deduce a closed-form formula for the Hausdorff distance between unit $\ell_{p_{1}}$ and $\ell_{p_{2}}$ norm balls in $\mathbb{R}^{d}$ for $1 \leq p_1 < p_2 \leq \infty$, i.e., a formula for $\delta\left(\mathbb{B}_{p_1}^d, \mathbb{B}_{p_2}^d\right)$. We provide details on the landscape of the corresponding nonconvex optimization objective. We also derive closed-form formula between the $k_1$ and $k_2$ unit $D$-norm balls in $\mathbb{R}^{d}$ for $1 \leq k_1 < k_2 \leq d$.
     
    \item We derive upper bound for Hausdorff distance between the common linear transforms of the $\ell_p$ and $\ell_q$ norm balls. This upper bound is a scaled $2\rightarrow q$ induced operator norm of the linear map, where $1\leq q \leq \infty$ and the scaling depends on both $p$ and $q$. We point out a class of linear maps for which the aforesaid closed-form formula for the Hausdorff distance is recovered, thereby broadening the applicability of the formula.
    
    \item Bringing together results from the random matrix theory literature, we provide upper bounds for the expected Hausdorff distance when the linear map is random with independent mean-zero entries for two cases: when the entries have magnitude less than unity, and when the entries are standard Gaussian.
    
    \item We provide certain generalization of the aforesaid formulation by considering the Hausdorff distance between two set-valued integrals. These integrals represent convex compact sets obtained by applying a parametric family of linear maps to the unit norm balls, and then taking the Minkowski sums of the resulting sets in a suitable limiting sense. We highlight an application for the same in computing the Hausdorff distance between the reach sets of a controlled linear dynamical system with unit norm ball-valued input uncertainties. 
\end{list}

The organization is as follows. In Sec. \ref{sec:NormBalls}, we consider the Hausdorff distance between unit norm balls for two cases: $\ell_p$ norm balls for different $p$, and $D$-norm balls parameterized by different parameter $k$. We discuss the landscape of the corresponding nonconvex optimization problem and derive closed-form formula for the Hausdorff distance. Sec. \ref{sec:ComposeWithLinear} considers the Hausdorff distance between the common linear transformation of different $\ell_p$ norm balls, and bounds the same when the linear map is either arbitrary or random. In Sec. \ref{sec:integral}, we consider an integral version of the problem considered in Sec. \ref{sec:ComposeWithLinear} and illustrate one application in controlled linear dynamical systems with set-valued input uncertainties where this structure appears. These results could be of independent interest.  

\textbf{Notations and preliminaries:} Most notations are introduced in situ. We use $\llbracket n\rrbracket:=\{1,2,\hdots,n\}$ to denote the set of natural numbers from $1$ to $n$. Boldfaced lowercase and boldfaced uppercase letters are used to denote the vectors and matrices, respectively. The symbol $\mathbb{E}$ denotes the mathematical expectation, ${\rm{card}}(\cdot)$ denotes the cardinality of a set, the superscript $^{\top}$ denotes matrix transpose, and the superscript $^{\dagger}$ denotes the appropriate pseudo-inverse. For a column vector $\bm{x}\in\mathbb{R}^{d}$ whose components are differentiable with respect to (w.r.t.) a scalar parameter $t$, the symbol $\dot{\bm{x}}$ denotes componentwise derivative of $\bm{x}$ w.r.t. $t$. The notation $\lfloor\cdot\rfloor$ stands for the floor function that returns the greatest integer less than or equal to its real argument. The function $\exp(\cdot)$ with matrix argument denotes the matrix exponential. The inequality $\succeq$ is to be understood in L\"{o}wner sense; e.g., saying $\bm{S}$ is a symmetric positive semidefinite matrix is equivalent to stating $\bm{S}\succeq \bm{0}$.

For any norm $\|\cdot\|$ in $\mathbb{R}^{d}$, its dual norm $\left(\|\cdot\|\right)^{*}$ is defined to be the support function of its unit norm ball, i.e.,
$$\left(\|\cdot\|\right)^{*}(\bm{y}) = \underset{\|\bm{x}\|\leq 1}{\sup}\langle \bm{y},\bm{x}\rangle.$$
The notation above emphasizes that the dual norm is a function of the vector $\bm{y}$. For $1\leq p \leq \infty$, it is well known that the dual of the $\ell_{p}$ norm is the $\ell_{q}$ norm where $q$ is the H\"{o}lder conjugate of $p$.

For $1\leq p, q \leq \infty$, matrix $\bm{M}\in\mathbb{R}^{m\times n}$ viewed as a linear map $\bm{M}:\ell_{p}\left(\mathbb{R}^{n}\right)\mapsto\ell_{q}\left(\mathbb{R}^{m}\right)$, has an associated induced operator norm
\begin{align}
\|\bm{M}\|_{p\rightarrow q} := \underset{\bm{x}\neq\bm{0}}{\sup}\dfrac{\|\bm{Mx}\|_{q}}{\|\bm{x}\|_{p}} = \underset{\|\bm{x}\|_{p}=1}{\sup}\|\bm{Mx}\|_{q},    
\label{defInducedOpNorm}    
\end{align}
where as usual $\|\bm{x}\|_{p} := \left(\sum_{i=1}^{n}\mid x_{i}\mid ^{p}\right)^{1/p}$, $\|\bm{Mx}\|_{q} := \left(\sum_{i=1}^{m}\mid (\bm{Mx})_{i}\mid ^{q}\right)^{1/q}$ for $p,q$ finite, $\|\cdot\|_{\infty}$ is the sup norm, and $(\bm{Mx})_{i}$ denotes the $i$th component of the vector $\bm{Mx}$. Several special cases of \eqref{defInducedOpNorm} are well known: the case $p=q$ is the standard matrix $p$ norm, the case $p=\infty, q=1$ is the Grothendieck problem \cite{grothendieck1956resume,alon2004approximating} that features prominently in combinatorial optimization, and its generalization $p\in(1,\infty), q=1$ is the $\ell_{p}$ Grothendieck problem \cite{kindler2010ugc}. In our development, the operator norm $\|\bm{M}\|_{2\rightarrow q}$ arises where $1 < q \leq \infty$. 

\section{Hausdorff Distance between Unit Norm Balls}\label{sec:NormBalls}
We consider the case when in \eqref{HausdorffSptFn}, the sets $\mathcal{K}_{1}\equiv \mathbb{B}^{d}_{p_{1}},\mathcal{K}_{2}\equiv \mathbb{B}^{d}_{p_{2}}$, the unit $\ell_{p_{1}}$ and $\ell_{p_{2}}$ norm balls in $\mathbb{R}^{d}$, $d\geq 2$, for $1 \leq p_1 < p_2 \leq \infty$. Clearly, the Hausdorff distance $\delta=0$ for $p_1=p_2$, and $\delta >0$ otherwise. Then the corresponding support functions $h_{1}(\cdot), h_{2}(\cdot)$ are the respective dual norms, i.e.,
$$h_{1}(\bm{y}) = \|\bm{y}\|_{q_{1}}, \quad h_{2}(\bm{y}) = \|\bm{y}\|_{q_{2}},\quad \frac{1}{p_1}+\frac{1}{q_{1}}=1, \quad \frac{1}{p_2}+\frac{1}{q_{2}}=1,$$
for  $1 \leq q_2 < q_1 \leq \infty$. By monotonicity of the norm function, we know that $\|\cdot\|_{q_{1}} \leq \|\cdot\|_{q_{2}}$. Therefore, the Hausdorff distance \eqref{HausdorffSptFn} in this case becomes
\begin{align}
\delta\left(\mathcal{K}_{1},\mathcal{K}_{2}\right) = \delta\left(\mathbb{B}_{p_1}^d, \mathbb{B}_{p_2}^d\right) = \underset{\|\bm{y}\|_{2} = 1}{\sup}\quad \left(\|\bm{y}\|_{q_{2}}- \|\bm{y}\|_{q_{1}}\right)\label{HausdorffbetnNormBallsFinal}    
\end{align}
which has a difference of convex (DC) objective. In fact, the objective is nonconvex (the difference of convex functions may or may not be convex in general) because it admits multiple global maximizers and minimizers.

The objective in \eqref{HausdorffbetnNormBallsFinal} is invariant under the plus-minus sign permutations among the components of the unit vector $\bm{y}$. There are $2^{d}$ such permutations feasible in $\mathbb{R}^{d}$ which implies that the landscape of the objective in \eqref{HausdorffbetnNormBallsFinal} has $2^{d}$ fold symmetry.  In other words, the feasible set is subdivided into $2^{d}$ sub-domains as per the sign permutations among the components of $\bm{y}$, and the``sub-landscapes" for these sub-domains are identical.

Since $\|\bm{y}\|_{q_1} \leq \|\bm{y}\|_{q_2}$ for $ 1 \leq q_2 < q_1 \leq \infty$, hence $0 \leq \delta$. The \emph{global minimum} value of the objective in \eqref{HausdorffbetnNormBallsFinal} is zero, which is achieved by any scaled basis vector, i.e., by $\bm{y}^{\rm{min}}:=\alpha\bm{e}_{k}\in\mathbb{R}^{d}$ for any $k\in\llbracket d\rrbracket$ and arbitrary $\alpha\in\mathbb{R}\setminus\{0\}$. These $\bm{y}^{\rm{min}}$ comprise uncountably many global minimizers for \eqref{HausdorffbetnNormBallsFinal}.

\begin{figure}%
    \centering
    \subfloat[Convex $\ell_p$ norm balls in $\mathbb{R}^{2}$.]{{\includegraphics[width=0.48\linewidth]{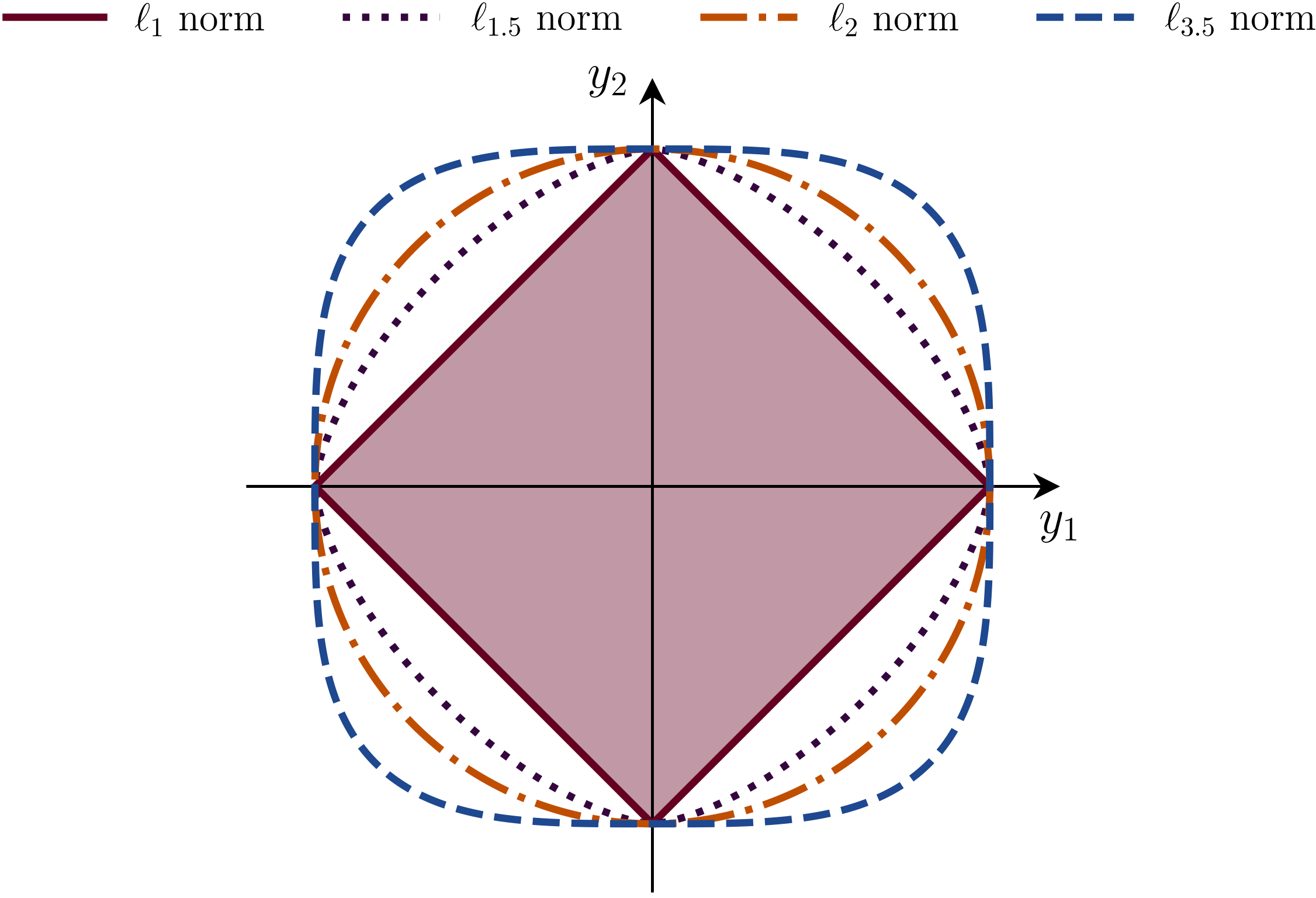}}}%
    \quad
    \subfloat[The Hausdorff distance \eqref{GlobalMaxNormBalls} for fixed $p_1=1$.]{{\includegraphics[width=0.47\linewidth]{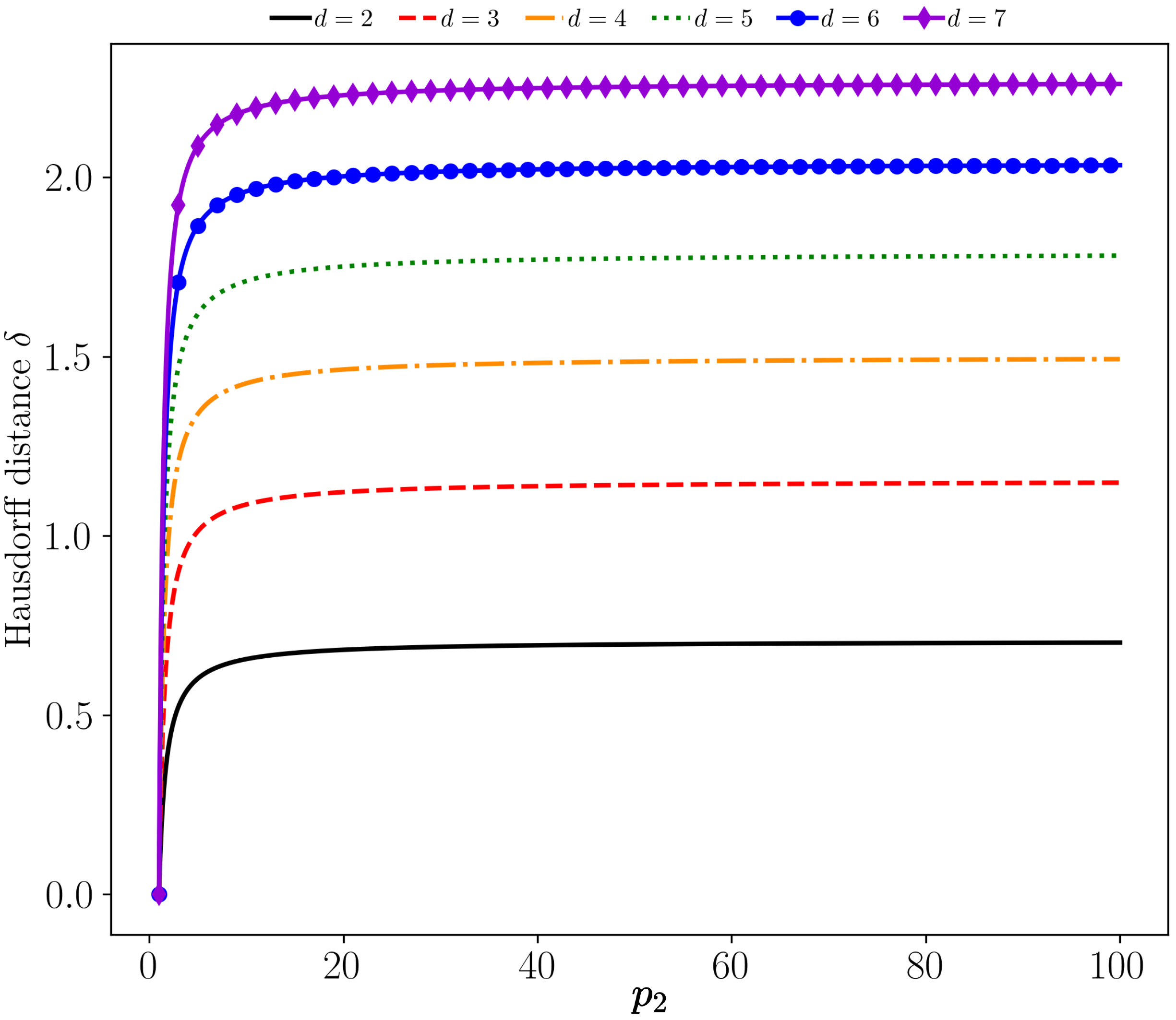} }}%
    \caption{Understanding the Hausdorff distance $\delta$ between the unit $\ell_{p_1}$ and $\ell_{p_2}$ norm balls in $\mathbb{R}^{d}$, $d\geq 2$, for  $1\leq p_1 < p_2 \leq \infty$.}%
    \label{fig:NormBalls}%
\end{figure}

We can compute the \emph{global maximum} value achieved in \eqref{HausdorffbetnNormBallsFinal} using the norm inequality
\begin{align}
\|\cdot\|_{q_{2}} \leq d^{\frac{1}{q_{2}} - \frac{1}{q_{1}}}\|\cdot\|_{q_{1}}, \quad 1 \leq q_2 < q_1 \leq \infty,
\label{NormUpperBound}
\end{align}
which follows from the H\"{o}lder's inequality: 
$$\displaystyle\sum_{i=1}^{d} \vert a_{i}b_{i}\vert \leq \left(\displaystyle\sum_{i=1}^{d} \vert a_i\vert^{r}\right)^{\frac{1}{r}} \left(\displaystyle\sum_{i=1}^{d} \vert b_i\vert^{\frac{r}{r-1}}\right)^{1 - \frac{1}{r}} \quad \bm{a},\bm{b}\in\mathbb{R}^{d}, \quad 1\leq r \leq \infty,$$
where the exponents $r$ and $\frac{r}{r-1}$ are H\"{o}lder conjugates. Applying this inequality with $\vert a_i \vert =\vert x_i\vert^{q_2}$, $\vert b_i\vert =1$, $r=q_1/q_2 > 1$, results in \eqref{NormUpperBound}.

In $\mathbb{R}^{d}$, the constant $d^{1/q_{2} - 1/q_{1}}$ is sharp because the equality in \eqref{NormUpperBound} is achieved by any vector in $\{-1,1\}^{d}$. Since \eqref{HausdorffbetnNormBallsFinal} has constraint $\|\bm{y}\|_{2}=1$, the corresponding global maximum will be achieved by 
$$\bm{y}^{\rm{\max}}\in\mathcal{Y}^{\max}:=\{\bm{u}\in\mathbb{S}^{d-1}\mid \bm{u} = \rho\bm{v}, \bm{v}\in\{-1,1\}^{d}, \rho>0\}.$$
The scalar $\rho$ is determined by the normalization constraint $\|\bm{y}^{\rm{\max}}\|_{2}=1$ as $\rho=1/\sqrt{d}$. Thus, we obtain 
\begin{align}
\delta &=\quad\underset{\|\bm{y}\|_{2} = 1}{\sup}\quad \left(\|\bm{y}\|_{q_{2}}- \|\bm{y}\|_{q_{1}}\right),   \qquad 1 \leq q_2 < q_1 \leq \infty,\nonumber\\
&= \left(d^{\frac{1}{q_{2}} - \frac{1}{q_{1}}} - 1\right)\underbrace{\|\bm{y}^{\max}\|_{q_{1}}}_{= \rho d^{\frac{1}{q_{1}}}}=d^{-\frac{1}{2}}\left(d^{\frac{1}{q_{2}}} - d^{\frac{1}{q_{1}}}\right),
\label{GlobalMaxNormBalls}
\end{align}
where in the last line we substituted $\rho=1/\sqrt{d}$. The cardinality of $\mathcal{Y}^{\max}$ equals $2^{d}$, i.e., there are $2^{d}$ global maximizers $\bm{y}^{\rm{\max}}\in\mathbb{S}^{d-1}$ achieving the value \eqref{GlobalMaxNormBalls}.

We summarize the above in the following Proposition.
\begin{proposition}\label{prop:HausdorffpNormBallsdifferentp}
For $1\leq p_1 < p_2 \leq \infty$, we have
\begin{align}
\delta\left(\mathbb{B}_{p_1}^d, \mathbb{B}_{p_2}^d\right) = d^{-\frac{1}{2}}\left(d^{\frac{1}{q_{2}}} - d^{\frac{1}{q_{1}}}\right),
\label{FinalFormulaHausdorffpNormBalls}
\end{align}
where $q_i$ denotes the H\"{o}lder conjugate of $p_i$ for $i\in\{1,2\}$.
\end{proposition}

\begin{remark}\label{remark:geometrynormballdelta}

As the intuition suggests, for a fixed $p_{1}$, larger $p_{2}$ results in a larger $\delta$ in a given dimension $d\geq 2$; see Fig. \ref{fig:NormBalls}.
\end{remark}

Fig. \ref{fig:NormBallContourPlot} shows the contour plot of $\|\bm{y}\|_{1}-\|\bm{y}\|_{2}$ in the spherical coordinates for $d=3$, i.e., $\bm{y}\in\mathbb{S}^{2}$. As predicted by \eqref{GlobalMaxNormBalls}, in this case, there are eight maximizers achieving the global maximum value $\sqrt{3}-1 \approx 0.7321$. The symmetric sub-landscapes mentioned earlier are also evident in Fig. \ref{fig:NormBallContourPlot}.

\begin{figure}%
    \centering
    \includegraphics[width=0.85\linewidth]{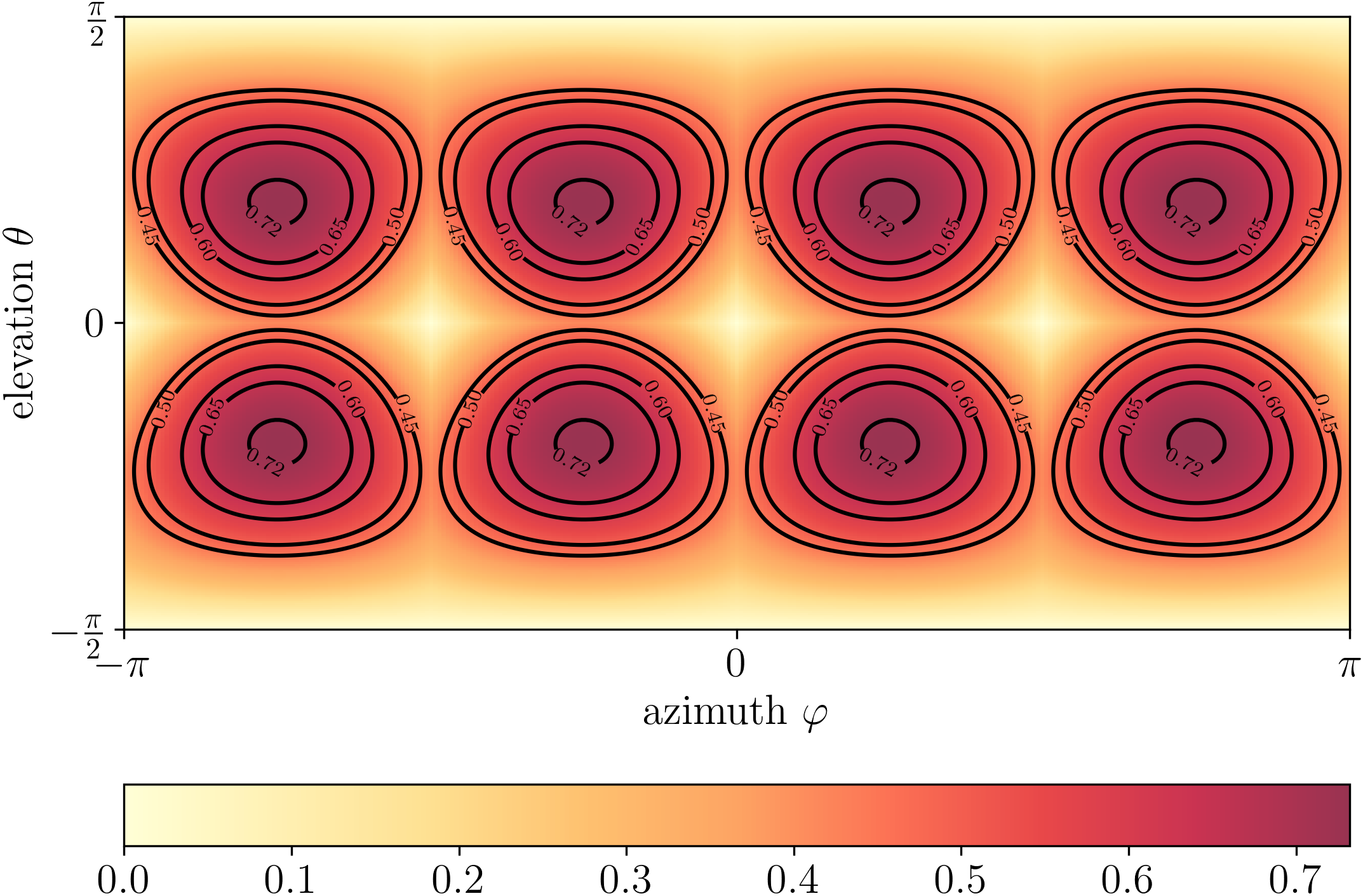}
    \caption{The landscape of the objective in \eqref{HausdorffbetnNormBallsFinal} for $d=3$, $q_1=2$, and $q_2=1$.}%
    \label{fig:NormBallContourPlot}%
\end{figure}

\subsection{Hausdorff Distance between Polyhedral $D$-Norm Balls} 
We next show that similar arguments as above can be used to derive the Hausdorff distance between other type of norm balls such as the $D$-norm balls which are certain polyhedral norm balls. The $D$-norms and norm balls arise naturally in robust optimization, see e.g., \cite[Sec. 2.2]{rahal2021norm}, \cite{bertsimas2004robust}. The $D$-norm in $\mathbb{R}^{d}$ is parameterized by $k$, where $1\leq k\leq d$, as defined next.
\begin{definition}(\textbf{$D$-norm})
For $1\leq k\leq  d$, the $D$-norm of $\bm{x}\in\mathbb{R}^{d}$ is
\begin{align}
\|\bm{x}\|_{k}^{D}:=\max _{\{S \cup\{t\} \mid S \subseteq \llbracket d\rrbracket, {\rm{card}}\left(S\right) \leq\lfloor k\rfloor, t \in \llbracket d\rrbracket \setminus S\}}\left\{\sum_{i \in S}\lvert x_{i} \rvert+(k-\lfloor k\rfloor)\lvert x_{t}\rvert \right\}.
\label{defDnorm}
\end{align}
\end{definition}
For $k=1$, the norm \eqref{defDnorm} reduces to the $\ell_{\infty}$ norm, i.e., $\|\bm{x}\|_{1}^{D} = \|\bm{x}\|_{\infty}$. For $k=d$, the norm \eqref{defDnorm} reduces to the $\ell_{1}$ norm, i.e., $\|\bm{x}\|_{d}^{D} = \|\bm{x}\|_{1}$. For $1<k<d$, the norm \eqref{defDnorm} can be thought of as a polyhedral interpolation between the $\ell_{\infty}$ and the $\ell_{1}$ norms. For a plot of the unit $D$-norm balls in $\mathbb{R}^{2}$, we refer the readers to \cite[Fig. 2]{rahal2021norm}.

A special case of \eqref{defDnorm} is when the parameter $k$ is restricted to be a natural number, i.e., $k\in\llbracket d\rrbracket$. Then the $D$-norm reduces to the so-called $k$ largest magnitude norm, defined next. 
\begin{definition}(\textbf{$k$ largest magnitude norm})
For $k\in\llbracket d\rrbracket$, the $k$ largest magnitude norm of $\bm{x}\in\mathbb{R}^{d}$ is
\begin{align}
\|\bm{x}\|_{[k]}:=\lvert x_{i_{1}}\rvert + \lvert x_{i_{2}}\rvert +\hdots+ \lvert x_{i_{k}}\rvert,
\label{defklargestmagnorm}
\end{align}
where the inequality $\lvert x_{i_{1}} \rvert \geq \lvert x_{i_{2}}\rvert \geq\hdots\geq \lvert x_{i_{d}}\rvert$ denotes the ordering of the magnitudes of the entries in $\bm{x}$.
\end{definition}
It is easy to verify that \eqref{defDnorm} (and thus its special case \eqref{defklargestmagnorm}) is indeed a norm, and its dual norm equals \cite[Prop. 2]{bertsimas2004robust}
$$\left(\|\cdot\|_{k}^{D}\right)^{*}(\bm{y}) = \max\bigg\{\dfrac{1}{k}\|\bm{y}\|_{1},\|\bm{y}\|_{\infty}\bigg\}.$$ 
For a comparison of the $D$-norm and its dual with the Euclidean norm, see \cite[Prop. 3]{bertsimas2004robust}. We have the following result.
\begin{proposition}\label{prop:ExactHausdorffkLargestMagNorm}(\textbf{Hausdorff distance $\delta$ between unit $D$-norm balls})
Let $1\leq k_1 < k_2 \leq d$, and let $\mathcal{K}_1,\mathcal{K}_2\subset\mathbb{R}^{d}$ denote the unit $\|\cdot\|_{k_{1}}^{D}$ and $\|\cdot\|_{k_{2}}^{D}$ norm balls in $\mathbb{R}^{d}$, i.e., $\mathcal{K}_1 \equiv \mathbb{B}_{\|\cdot\|_{k_1}^D}^{d}, \mathcal{K}_2 \equiv \mathbb{B}_{\|\cdot\|_{k_2}^D}^{d}$. Then
\begin{align}
\delta\left(\mathcal{K}_1,\mathcal{K}_2\right) = \delta\left(\mathbb{B}_{\|\cdot\|_{k_1}^D}^{d},\mathbb{B}_{\|\cdot\|_{k_2}^D}^{d}\right) = \left(\dfrac{1}{k_1}-\dfrac{1}{k_2}\right)\sqrt{d}.
\label{ExactFormulaHausdorffklargestmagnormdiff}
\end{align}
\end{proposition}
\begin{proof}
Let $h_1(\cdot),h_2(\cdot)$ denote the support functions of $\mathcal{K}_1,\mathcal{K}_2$, respectively. Using the definition of dual norm, we have
\begin{align}
h_{1}(\bm{y}) = \max\bigg\{\dfrac{1}{k_1}\|\bm{y}\|_{1},\|\bm{y}\|_{\infty}\bigg\}, \quad h_{2}(\bm{y}) = \max\bigg\{\dfrac{1}{k_2}\|\bm{y}\|_{1},\|\bm{y}\|_{\infty}\bigg\}.    
\label{SptFnklargestmag}    
\end{align}
Recall that $\delta$ relates to $h_1,h_2$ via \eqref{HausdorffSptFn}. Since $1\leq k_1 < k_2 \leq d$, we know that $\dfrac{1}{k_2}\|\bm{y}\|_{1} < \dfrac{1}{k_1}\|\bm{y}\|_{1}$. Depending on the value of $\|\bm{y}\|_{\infty}$, we need to consider three subsets of unit vectors.

Specifically, for the unit vectors $\bm{y}$ satisfying $\dfrac{1}{k_2}\|\bm{y}\|_{1} < \dfrac{1}{k_1}\|\bm{y}\|_{1} \leq \|\bm{y}\|_{\infty}$, we have $h_1(\bm{y})=h_{2}(\bm{y})=\|\bm{y}\|_{\infty}$ and $h_1(\bm{y})-h_2(\bm{y})=0$.

On the other hand, for the unit vectors $\bm{y}$ satisfying $\|\bm{y}\|_{\infty}\leq \dfrac{1}{k_2}\|\bm{y}\|_{1} < \dfrac{1}{k_1}\|\bm{y}\|_{1}$, we have $h_1(\bm{y}) = \dfrac{1}{k_1}\|\bm{y}\|_{1}$, $h_2(\bm{y}) = \dfrac{1}{k_2}\|\bm{y}\|_{1}$, hence we obtain that $h_1(\bm{y})-h_2(\bm{y})=(1/k_1 - 1/k_2)\|\bm{y}\|_{1}$, which is always nonnegative.

Finally, for the unit vectors $\bm{y}$ satisfying $\dfrac{1}{k_2}\|\bm{y}\|_{1} \leq \|\bm{y}\|_{\infty}  < \dfrac{1}{k_1}\|\bm{y}\|_{1}$, we must have $h_1(\bm{y})-h_2(\bm{y})=\dfrac{1}{k_1}\|\bm{y}\|_{1} - \|\bm{y}\|_{\infty} < (1/k_1 - 1/k_2)\|\bm{y}\|_{1}$. 

Therefore, using \eqref{HausdorffSptFn} we get
$$\delta\left(\mathcal{K}_1,\mathcal{K}_2\right) = \left(\dfrac{1}{k_1}-\dfrac{1}{k_2}\right)\underset{\|\bm{y}\|_2 = 1}{\sup}\|\bm{y}\|_{1}.$$
Using the same arguments as in \eqref{GlobalMaxNormBalls}, we obtain $\underset{\|\bm{y}\|_2 = 1}{\sup}\|\bm{y}\|_{1} = \sqrt{d}$, which is achieved by $2^{d}$ vectors of the form $\bm{v}/\sqrt{d}$ for all $\bm{v}\in\{-1,1\}^{d}$. This completes the proof.\end{proof}

\begin{figure}%
    \centering
    \includegraphics[width=0.85\linewidth]{ 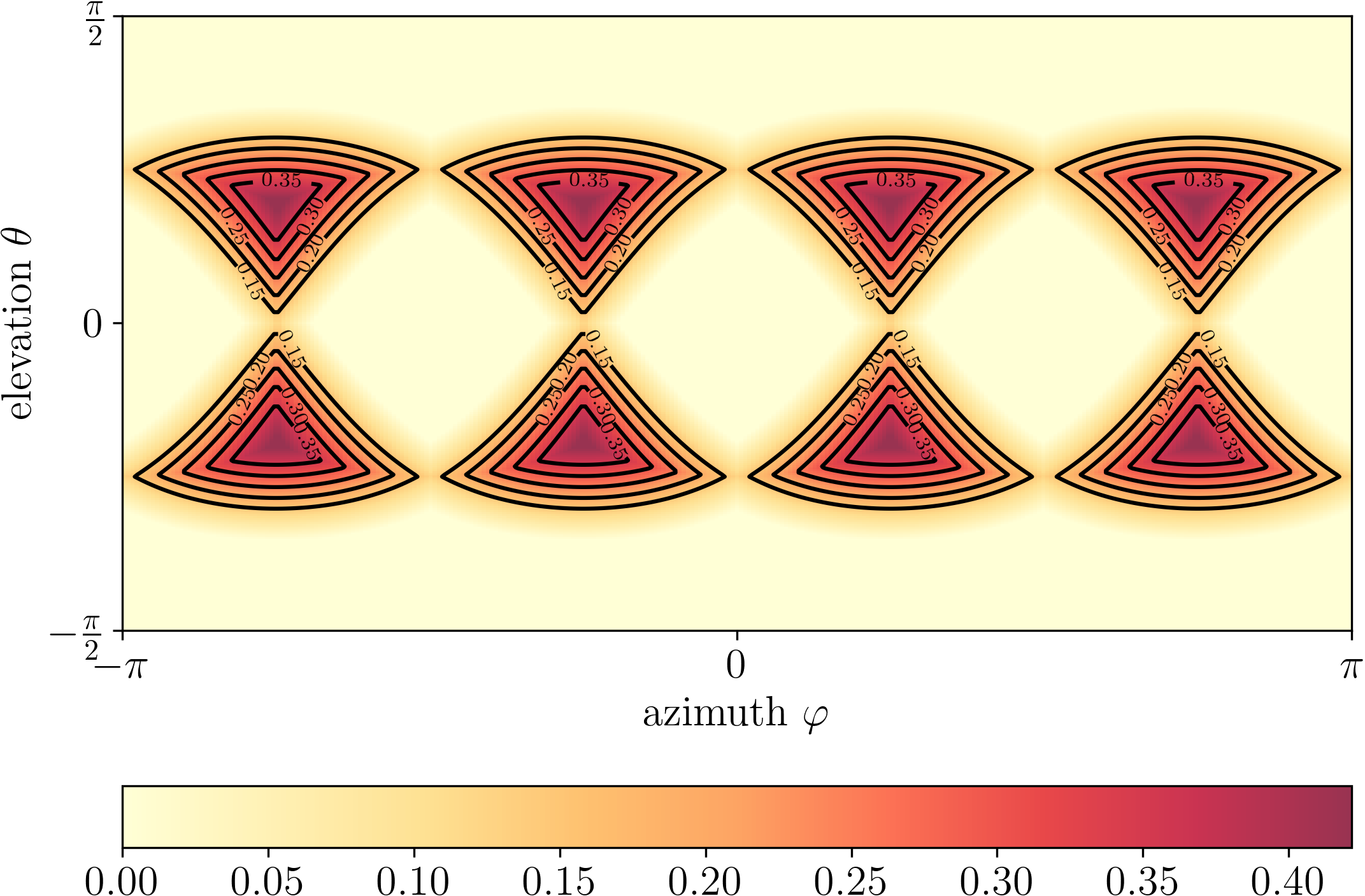}
    \caption{The landscape of the objective in \eqref{HausdorffSptFn} with $h_1,h_2$ as in \eqref{SptFnklargestmag} for $d=3$, $k_1=1.7$, and $k_2=2.9$.}%
    \label{fig:klargestNormBallContourPlot}%
\end{figure}

Fig. \ref{fig:klargestNormBallContourPlot} shows the landscape of the objective for computing the Hausdorff distance between the unit $D$-norm balls with $k_1=1.7$ and $k_2=2.9$ in $d=3$ dimensions, and as explained in the proof above, there are eight global maximizers given by $\bm{v}/\sqrt{3}$ for all $\bm{v}\in\{-1,1\}^{3}$. In this case, the formula \eqref{ExactFormulaHausdorffklargestmagnormdiff} gives $\delta=120\sqrt{3}/493 \approx 0.421594517055305$ while the direct numerical estimate of $\delta$ from the contours yields $0.421577951149235$.

\section{Composition with a Linear Map}\label{sec:ComposeWithLinear}
We next consider a generalized version of \eqref{HausdorffbetnNormBallsFinal} given by
\begin{align}
\label{HausdorffbetnNormBallsComposedWithLinear}  
\delta\left(\mathcal{K}_{1},\mathcal{K}_{2}\right) = \underset{\|\bm{y}\|_{2} = 1}{\sup}\quad \left(\|\bm{T}\bm{y}\|_{q_{2}}- \|\bm{T}\bm{y}\|_{q_{1}}\right),  \qquad 1 \leq q_{2} < q_{1} \leq \infty,    
\end{align}
where the matrix $\bm{T}\in\mathbb{R}^{m\times d}$, $m\leq d$, has full row rank $m$. Using \eqref{SptFnLinearMap}, we can interpret \eqref{HausdorffbetnNormBallsComposedWithLinear} as follows. As before, let $p_1,p_2$ denote the H\"{o}lder conjugates of $q_1,q_2$, respectively. Then \eqref{HausdorffbetnNormBallsComposedWithLinear} computes the Hausdorff distance between two compact convex sets in $\mathbb{R}^{d}$ obtained as the linear transformations of the $m$-dimensional $\ell_{p_{1}}$ and $\ell_{p_{2}}$ unit norm balls via $\bm{T}^{\top}\in\mathbb{R}^{d\times m}$, i.e., $$\mathcal{K}_{1}\equiv \bm{T}^{\top} \mathbb{B}^{d}_{p_{1}}, \quad\mathcal{K}_{2}\equiv \bm{T}^{\top} \mathbb{B}^{d}_{p_2}.$$
Since the right pseudo-inverse  $\bm{T}^{\dagger} = \bm{T}^{\top}\left(\bm{T}\bm{T}^{\top}\right)^{-1}$, one can equivalently view \eqref{HausdorffbetnNormBallsComposedWithLinear} as that of maximizing the difference between the $\ell_{p_{1}}$ and $\ell_{p_{2}}$ norms over the $m$-dimensional origin-centered ellipsoid with shape matrix $\bm{TT}^{\top}$.

As was the case in \eqref{HausdorffbetnNormBallsFinal}, problem \eqref{HausdorffbetnNormBallsComposedWithLinear} is a DC programming problem with nonconvex objective. However, unlike \eqref{HausdorffbetnNormBallsFinal}, now there is no obvious symmetry in the objective's landscape that can be leveraged because the number and locations of the local maxima or saddles have sensitive dependence on the matrix parameter $\bm{T}$; see the first column of Table \ref{table:1}. Thus, directly using off-the-shelf solvers such as \cite{dccpGitRepo,shen2016disciplined} or nonconvex search algorithms become difficult for solving \eqref{HausdorffbetnNormBallsComposedWithLinear} in practice as the iterative search may get stuck in a local stationary point.

\begin{remark}\label{remark:linearmapofDnorm}
We can also consider the Hausdorff distance between the common linear transforms of different polyhedral $D$ norm balls discussed earlier. Specifically, if $\mathcal{K}_1, \mathcal{K}_2$ are the unit $\|\cdot\|_{k_1}^{D}, \|\cdot\|_{k_2}^{D}$ norm balls for $1\leq k_1 < k_2 \leq d$, then following \eqref{SptFnLinearMap} and the same steps in the proof of Proposition \ref{prop:ExactHausdorffkLargestMagNorm}, the Hausdorff distance $\delta$ between the sets $\bm{T}\mathcal{K}_1,\bm{T}\mathcal{K}_2$ equals
\begin{align}
\delta\left(\bm{T}\mathcal{K}_1,\bm{T}\mathcal{K}_2\right) = \left(\dfrac{1}{k_1} - \dfrac{1}{k_2}\right)\underbrace{\underset{\|\bm{y}\|_2 = 1}{\sup}\|\bm{T^{\top}y}\|_{1}}_{=:\|\bm{T}^{\top}\|_{2\rightarrow 1}} = \left(\dfrac{1}{k_1} - \dfrac{1}{k_2}\right)\|\bm{T}\|_{\infty\rightarrow 2},
\label{HausdorffbetnLinearMapofDnormBalls}    
\end{align}
where the last equality follows from the relation between the induced norm of an operator and that of its adjoint.
\end{remark}
\subsection{Estimates for Arbitrary {\boldmath $T$}} 
We next provide an upper bound for  \eqref{HausdorffbetnNormBallsComposedWithLinear} in terms of the operator norm $\|\bm{T}\|_{2\rightarrow q_{1}}$.
\begin{proposition}\label{prop:upperbound} (\textbf{Upper bound})
Let $\bm{T}\in\mathbb{R}^{m\times d}$. Then for $1 \leq q_{2} < q_{1} \leq \infty$, we have 
\begin{align}
\underset{\|\bm{y}\|_{2}=1}{\sup}\left(\|\bm{Ty}\|_{q_{2}} - \|\bm{Ty}\|_{q_{1}}\right) \leq \left(m^{\frac{1}{q_{2}}-\frac{1}{q_{1}}}-1\right) \|\bm{T}\|_{2\rightarrow q_{1}}.\label{UpperBoundLinearComposition} \end{align}
\end{proposition}
\begin{proof}
Proceeding as in Sec. \ref{sec:NormBalls}, for $\bm{y}\in\mathbb{S}^{d-1}$ we get
\begin{align}
\|\bm{Ty}\|_{q_{2}} &\leq m^{\frac{1}{q_{2}}-\frac{1}{q_{1}}} \|\bm{Ty}\|_{q_{1}}\nonumber\\
\Rightarrow \|\bm{Ty}\|_{q_{2}} - \|\bm{Ty}\|_{q_{1}} &\leq \left(m^{\frac{1}{q_{2}}-\frac{1}{q_{1}}}-1\right)\|\bm{Ty}\|_{q_{1}}\nonumber\\
&\leq \left(m^{\frac{1}{q_{2}}-\frac{1}{q_{1}}}-1\right)\underset{\|\bm{y}\|_{2}=1}{\sup}\|\bm{Ty}\|_{q_{1}} = \left(m^{\frac{1}{q_{2}}-\frac{1}{q_{1}}}-1\right) \|\bm{T}\|_{2\rightarrow q_{1}}\nonumber\\
&\kern-8.5em\Rightarrow \underset{\|\bm{y}\|_{2}=1}{\sup}\left(\|\bm{Ty}\|_{q_{2}} - \|\bm{Ty}\|_{q_{1}}\right) \leq \left(m^{\frac{1}{q_{2}}-\frac{1}{q_{1}}}-1\right) \|\bm{T}\|_{2\rightarrow q_{1}}.\nonumber
\end{align}
\end{proof}
Recall that $1 < q_{1} \leq \infty$. When $1 \leq q_{2} < q_{1} \leq 2$, the operator norm $\|\bm{T}\|_{2\rightarrow q_{1}}$ is, in general, NP hard to compute \cite{hendrickx2010matrix,steinberg2005computation,bhaskara2011approximating} except in the well-known case $q_1=2$ for which $\|\bm{T}\|_{2\rightarrow 2} = \sigma_{\max}(\bm{T})$, the maximum singular value of $\bm{T}$. When $1 \leq q_{2} \leq 2 < q_{1} \leq \infty$, the norm $\|\bm{T}\|_{2\rightarrow q_{1}}$ is often referred to as hypercontractive \cite{barak2012hypercontractivity}, and its computation for generic $\bm{T}\in\mathbb{R}^{m\times d}$ is relatively less explored (see e.g., \cite{barak2012hypercontractivity,bhattiprolu2019approximability}) except for the case $q_{1}=\infty$ for which $\bm{T}_{2\rightarrow\infty} = \max_{i=1,\hdots,m}\|\bm{T}(i,:)\|_2$ (maximum $\ell_2$ norm of a row). Hypercontractive norms and related inequalities find applications in establishing rapid mixing of random walks as well as several problems of interest in theoretical computer science \cite{gross1975logarithmic,saloff1997lectures,biswal2011hypercontractivity,barak2012hypercontractivity}.

Table \ref{table:1} reports our numerical experiments to estimate \eqref{HausdorffbetnNormBallsComposedWithLinear} with $q_{1}=2,q_{2}=1$, for five random realizations of $\bm{T}\in\mathbb{R}^{3\times 3}$, arranged as the rows of Table \ref{table:1}. For visual clarity, the contour plots in the first column of Table \ref{table:1} depict only four high-magnitude contour levels. These results suggest that the landscape of the nonconvex objective in \eqref{HausdorffbetnNormBallsComposedWithLinear} has sensitive dependence on the mapping $\bm{T}$.

We can say more for specific classes of $\bm{T}$. For example, notice from \eqref{UpperBoundLinearComposition} that if the mapping $\bm{T}:\ell_{q}(\mathbb{R}^{d})\mapsto\ell_{q}\left(\mathbb{R}^{m}\right)$ is an isometry, i.e., $\|\bm{Ty}\|_{q} = \|\bm{y}\|_{q}$, then the upper bound is achieved by any $\bm{y}\in\mathbb{R}^{d}$ such that $\sqrt{d}\bm{y}\in\{-1,1\}^{d}$ as in Sec. \ref{sec:NormBalls}, and we recover the exact formula \eqref{FinalFormulaHausdorffpNormBalls}. We can characterize these isometric maps as follows.

\begin{table}
\centering
\begin{tabular}{|P{0.43\textwidth}|P{0.2\textwidth}|P{0.2\textwidth}|}
 \hline
 Landscape & Estimated max. & Upper bound \eqref{UpperBoundLinearComposition} \\ 
 \hline\hline
 
\raisebox{-0.6\totalheight}{\includegraphics[width=0.42\textwidth]{ 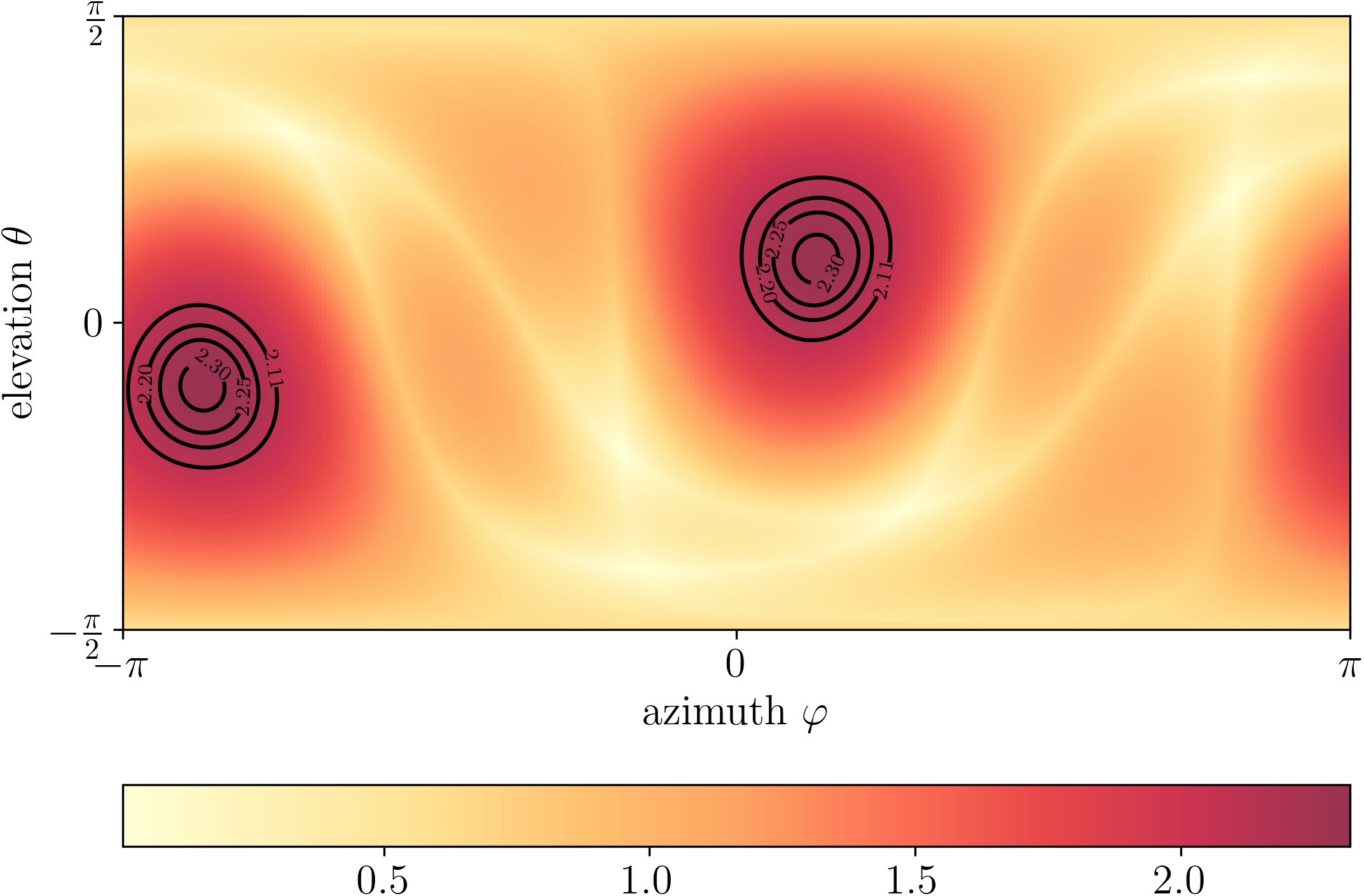}} & 2.318732079842860 & 2.384527280099902 \rule[-14ex]{0pt}{0pt}\\
  \hline
\raisebox{-0.6\totalheight}{\includegraphics[width=0.42\textwidth]{ 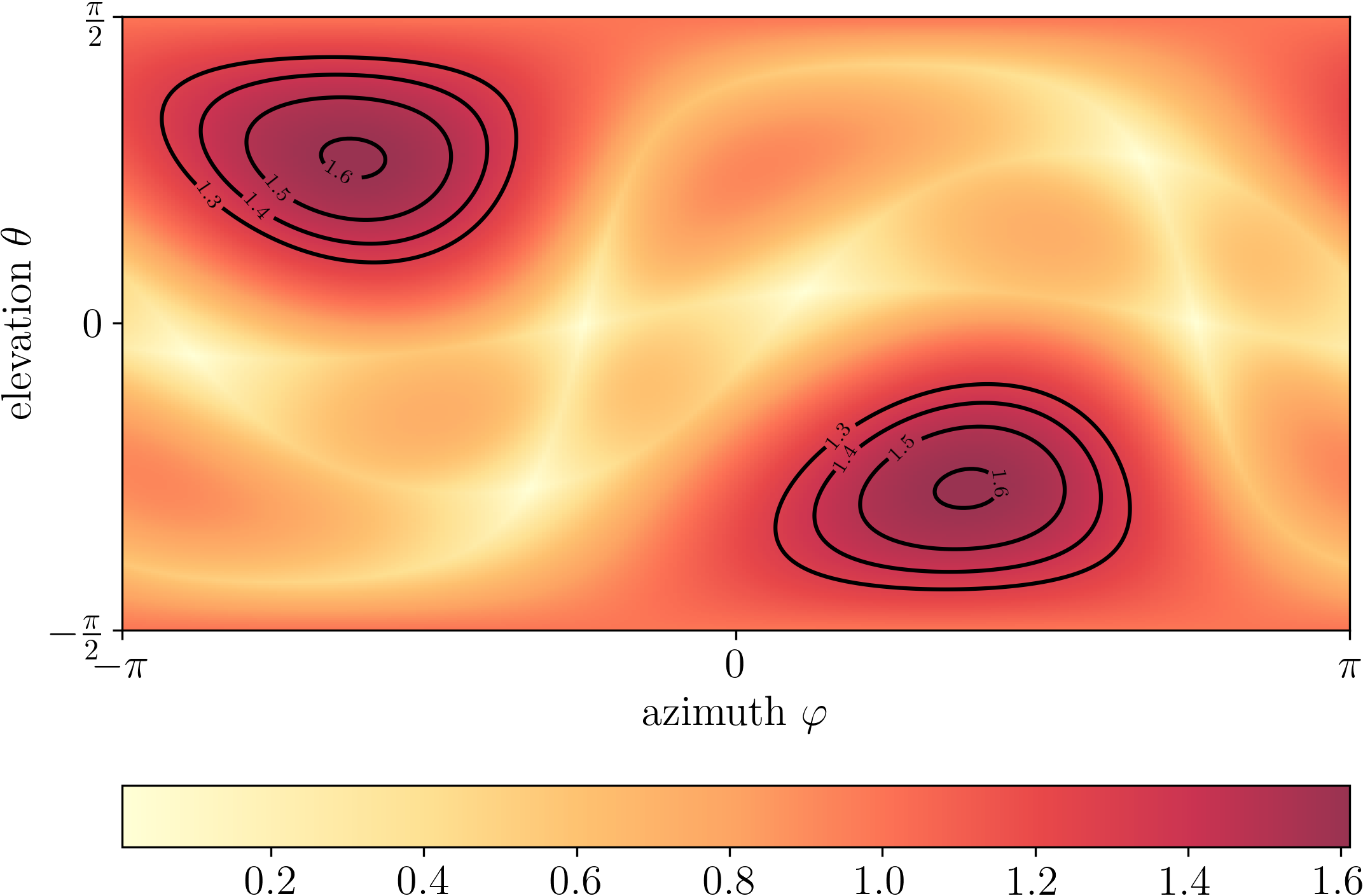}}  & 1.611342375325351 & 1.722680031455154 \rule[-14ex]{0pt}{0pt}\\
  \hline
\raisebox{-0.6\totalheight}{\includegraphics[width=0.42\textwidth]{ 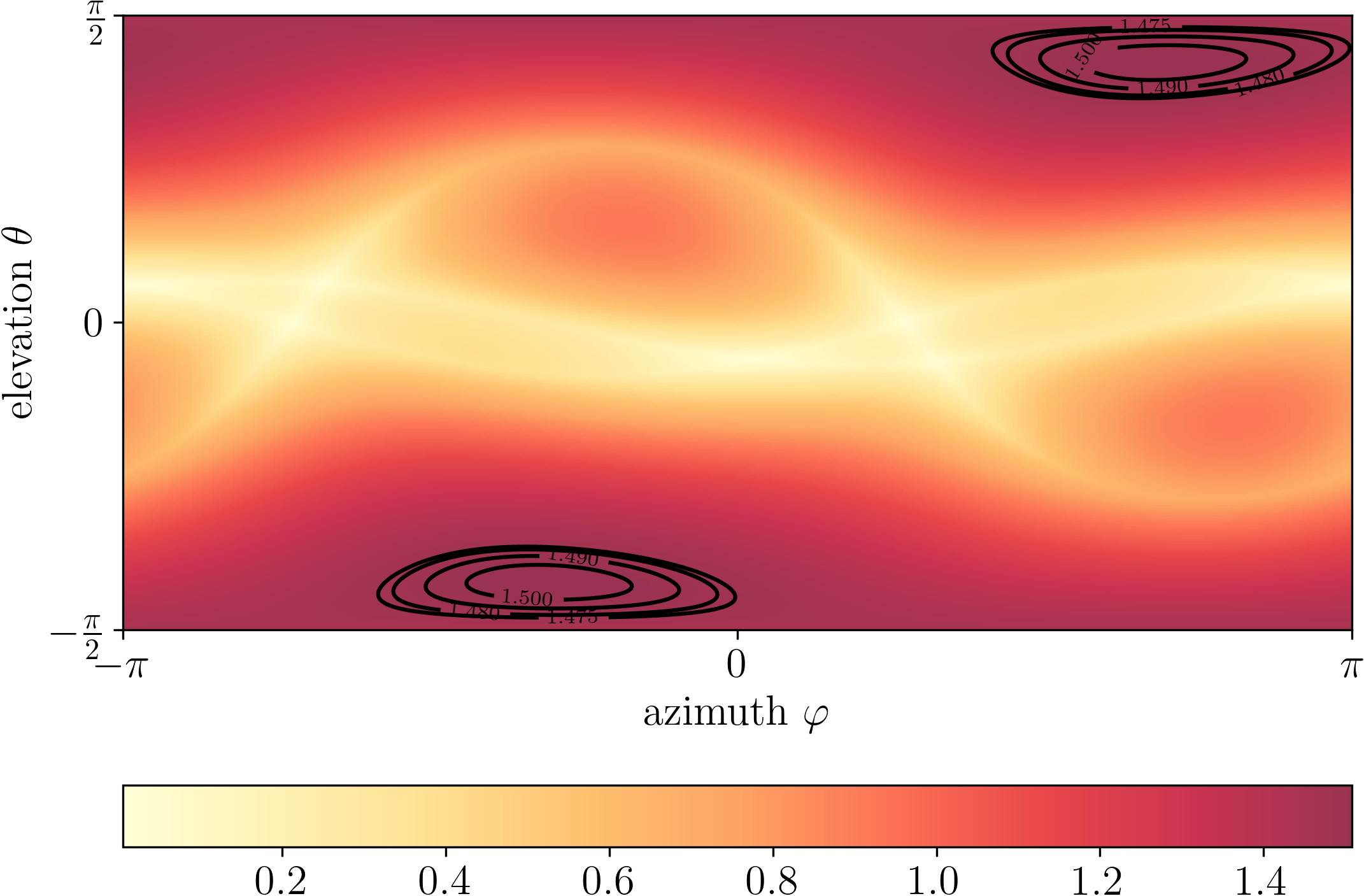}}  & 1.507938701982408 & 1.583409927359882 \rule[-14ex]{0pt}{0pt}\\
  \hline
 \raisebox{-0.6\totalheight}{\includegraphics[width=0.42\textwidth]{ 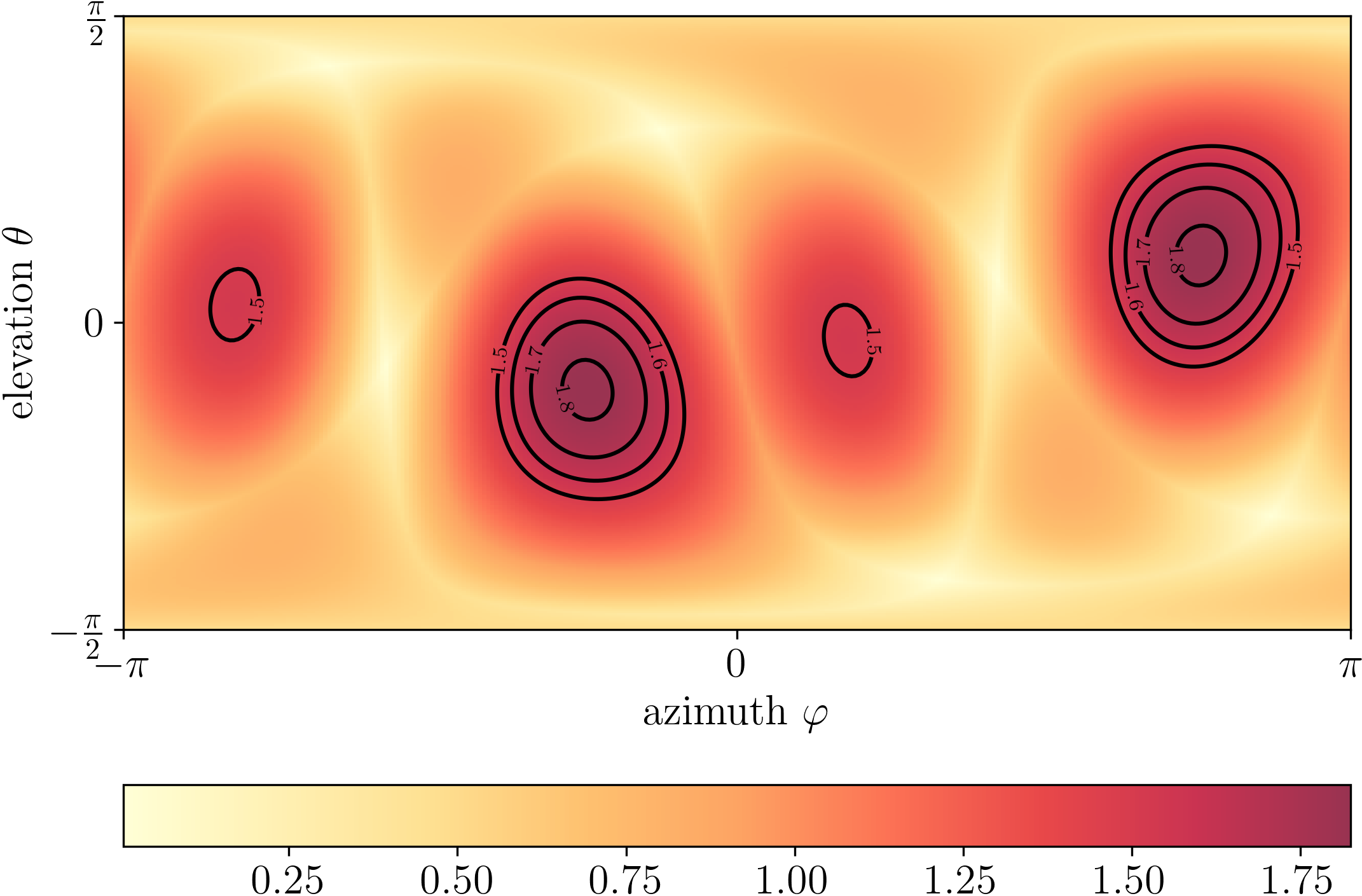}} & 1.824182821725298 & 2.157801577048147 \rule[-14ex]{0pt}{0pt}\\
 \hline
\raisebox{-0.6\totalheight}{\includegraphics[width=0.42\textwidth]{ 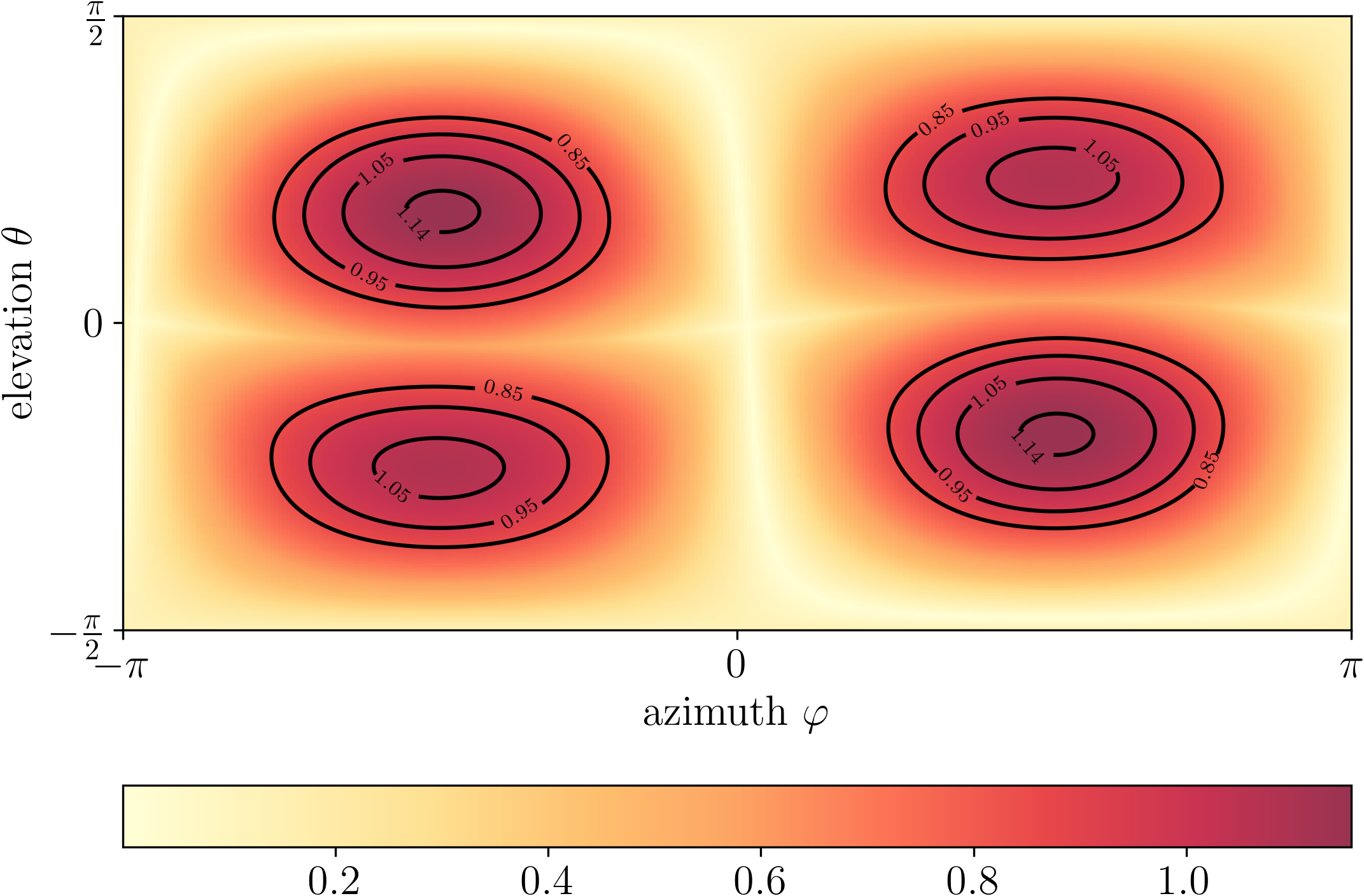}} & 1.154650461995163 & 1.303457527919837\rule[-16.5ex]{0pt}{0pt}\\ 
\hline
\end{tabular}
\caption{Landscapes of $\|\bm{Ty}\|_{q_{2}}-\|\bm{Ty}\|_{q_{1}}$ for $q_{1}=2$, $q_{2}=1$, $\bm{y}\in\mathbb{S}^{2}$ in spherical coordinates for five randomly generated $\bm{T}\in\mathbb{R}^{3\times 3}$ with independent standard Gaussian entries. The middle column reports the numerically estimated global maxima from the respective contour data, i.e., the estimated Hausdorff distance \eqref{HausdorffbetnNormBallsComposedWithLinear}. The last column shows the corresponding bounds \eqref{UpperBoundLinearComposition}.}
\vspace*{-0.1in}
\label{table:1}
\end{table}

\begin{proposition}\label{prop:isometry} (\textbf{Isometry})
Consider a linear mapping given by $\bm{T}:\ell_{q}(\mathbb{R}^{d})\mapsto\ell_{q}\left(\mathbb{R}^{m}\right)$.\\
\vspace*{0.01in}\\
(i) (See e.g., \cite[Remark 3.1]{wang1993structures}) For $q=2$, the mapping $\bm{T}\in\mathbb{R}^{m\times d}$ is an isometry if and only if $\bm{T}^{\top}\bm{T}=\bm{I}_{d}$, i.e., $\bm{T}$ is a column-orthonormal matrix.\\
\vspace*{0.01in}\\
(ii) (\cite[Thm. 3.2]{wang1993structures}) For $q\in[1,\infty)\setminus\{2\}$, the mapping $\bm{T}\in\mathbb{R}^{m\times d}$ is an isometry if and only if there exists a permutation matrix $\bm{P}\in\mathbb{R}^{m\times m}$ such that $\bm{PT}={\rm{diag}}(\bm{r}_1,\bm{r}_2,\hdots,\bm{r}_d)$ and $\|\bm{r}_j\|_q=1$ for all $j\in[d]$. In particular, when $d=m$ and $q\in[1,\infty)\setminus\{2\}$, the mapping $\bm{T}$ is isometry if and only if it is a signed permutation matrix \cite{li1994isometries,288084}, i.e, a permutation matrix whose nonzero entries are either all $+1$ or all $-1$ or some $+1$ and the rest $-1$.
\end{proposition}
The following is an immediate consequence of this characterization.
\begin{corollary}\label{corr:isometry}
For $\bm{T}$ as in Proposition \ref{prop:isometry}, the Hausdorff distance $\delta$ in \eqref{HausdorffbetnNormBallsComposedWithLinear} equals \eqref{FinalFormulaHausdorffpNormBalls}.
\end{corollary}

An instance in which $\|\bm{T}\|_{2\rightarrow q_{1}}$ and hence the bound \eqref{UpperBoundLinearComposition} is efficiently computable, occurs  when $\bm{T}\in\mathbb{R}^{m\times d}$ is elementwise nonnegative and $1\leq q_{1}<2$. In this case, the operator norm $\|\bm{T}\|_{2\rightarrow q_{1}}$ is known \cite[Thm. 3.3]{steinberg2005computation} to be equal to the optimal value of the following convex optimization problem:
\begin{align}
{\rm{OPT}} :=\; &\underset{\bm{X}\succeq\bm{0}}{\max}\:\sqrt{\|{\rm{dg}}\left(\bm{TXT}^{\top}\right)\|_{\frac{q_{1}}{2}}}\nonumber\\
&\text{subject to}\; \|{\rm{dg}}\left(\bm{X}\right)\|_{1} \leq 1,
\label{ConvexProblem}    
\end{align}
where ${\rm{dg}}\left(\cdot\right)$ takes a square matrix as its argument and returns the vector comprising of the diagonal entries of that matrix. To see why problem \eqref{ConvexProblem} is convex, notice that $\bm{X}\succeq\bm{0}$ has unique (principal) square root, so $\bm{TXT}^{\top} = \bm{TX}^{\frac{1}{2}}\left(\bm{TX}^{\frac{1}{2}}\right)^{\top} \succeq 0$ which implies ${\rm{dg}}\left(\bm{TXT}^{\top}\right)$ has nonnegative entries. Consequently, the objective in \eqref{ConvexProblem} is concave for $1\leq q_{1}<2$. The non-empty feasible set 
$\{\bm{X}\in\mathbb{R}^{d\times d} \mid \bm{X}\succeq\bm{0}, \: \|{\rm{dg}}\left(\bm{X}\right)\|_{1}=\sum_{i=1}^{d}X_{ii} \leq 1\}$
is the intersection of the positive semidefinite cone with a linear inequality, hence convex (in fact a spectrahedron).

Then, the right hand side of \eqref{UpperBoundLinearComposition} equals $\left(m^{\frac{1}{q_{2}}-\frac{1}{q_{1}}} - 1\right)\times {\rm{OPT}}$. For example, when 
\begin{align}
q_{1} = 1.5, \quad q_{2}=1, \quad \bm{T} = \begin{bmatrix}
2 & 6 & 0\\
5 & 0 & 1
\end{bmatrix},
\label{DataforCVX}    
\end{align}
a numerical solution of \eqref{ConvexProblem} via {\texttt{cvx}} \cite{cvx} gives ${\rm{OPT}}\approx 7.425702405524379$. As in Table 1, a direct numerical search over the nonconvex landscape (Fig. \ref{fig:cvxContourPlot}) for this example returns the estimated Hausdorff distance $\approx 1.888517738190415$ while using the numerically computed ${\rm{OPT}}$, we find the upper bound \eqref{UpperBoundLinearComposition} $\approx 1.930096365450782$.

\begin{remark}\label{remark:rangeofqforconvexopt}
We clarify here that for \eqref{ConvexProblem} to be used in the upper bound \eqref{UpperBoundLinearComposition}, the range of $q_1$ is $1 < q_1 < 2$. That $\|\bm{T}\|_{2\rightarrow q_{1}}$ equals to \eqref{ConvexProblem} holds also for the case $q_{1}=1$. Indeed, this implies we can compute \eqref{HausdorffbetnLinearMapofDnormBalls} for elementwise nonnegative $\bm{T}$ by computing $\|\bm{T}^{\top}\|_{2\rightarrow 1}$ via convex optimization.
\end{remark}

\begin{figure}%
    \centering
    \includegraphics[width=0.85\linewidth]{ 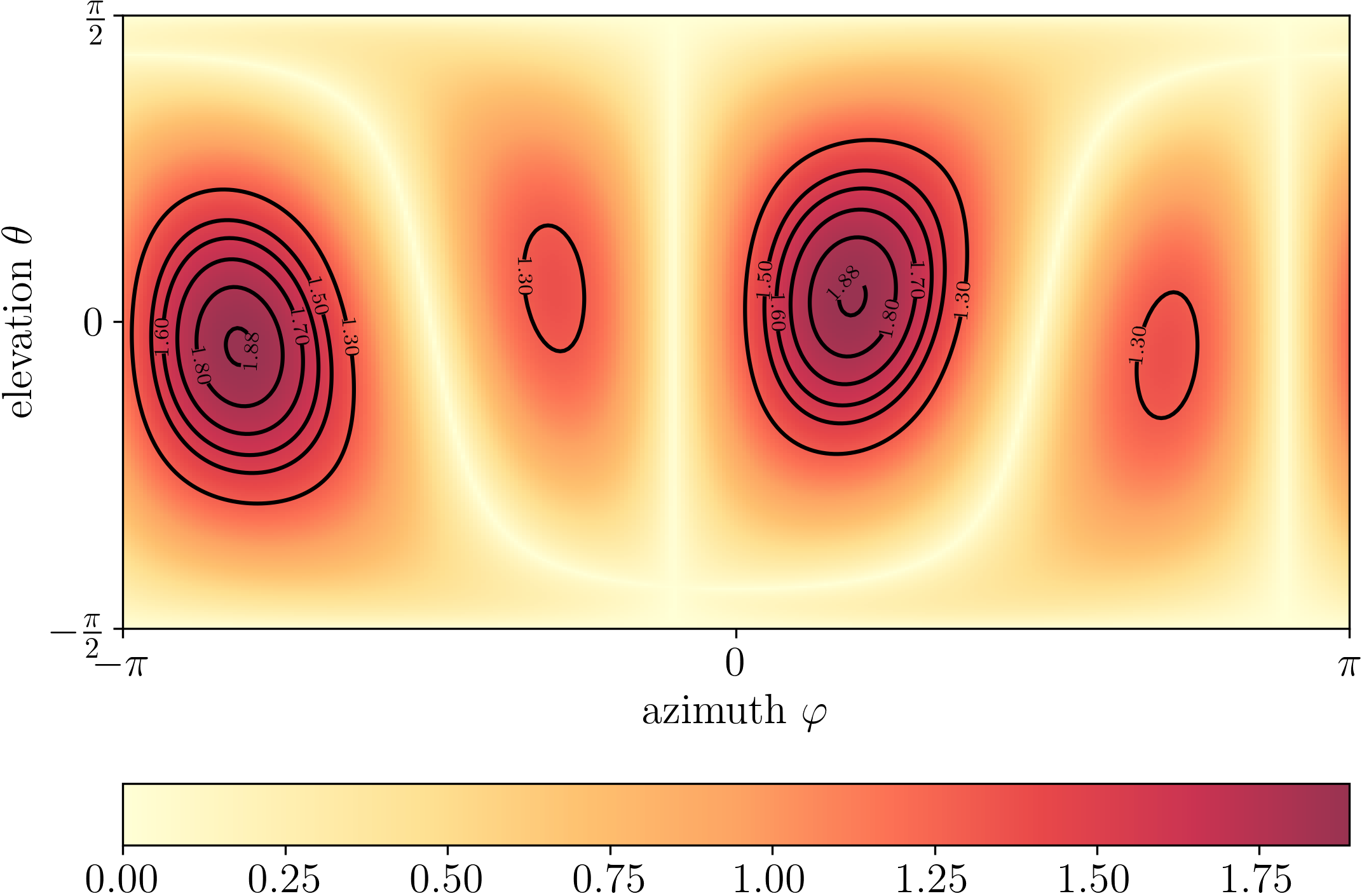}
    \caption{The landscape of the objective in \eqref{HausdorffbetnNormBallsComposedWithLinear} depicted in spherical coordinates for the problem data given in \eqref{DataforCVX}.}%
    \label{fig:cvxContourPlot}%
\end{figure}

\vspace*{0.1in}

\subsection{Estimates for Random {\boldmath $T$}} 
For random linear maps $\bm{T}:\ell_{q}(\mathbb{R}^{d})\mapsto\ell_{q}\left(\mathbb{R}^{m}\right)$, it is possible to bound the expected Hausdorff distance \eqref{HausdorffbetnNormBallsComposedWithLinear}. We collect two such results in the following proposition.

\begin{proposition}\label{prop:RandomLinearMap} (\textbf{Bound for the expected Hausdorff distance})\\
Let $2\leq q_{1} <\infty$.
\vspace*{0.01in}\\
(i) Let $\bm{T}=(\theta_{ij})_{i,j=1}^{m,d}$ have independent (not necessarily identically distributed) mean-zero entries with $\lvert \theta_{ij}\rvert\leq 1$ for all index pair $(i,j)$. Then the Hausdorff distance \eqref{HausdorffbetnNormBallsComposedWithLinear} satisfies
\begin{align}
\mathbb{E}\:\delta\leq \left(m^{\frac{1}{q_{2}}-\frac{1}{q_{1}}}-1\right)C_{q_{1}} \max\{m^{\frac{1}{q_{1}}},\sqrt{d}\}
\label{ExpectedHausdorffAijBetweenMinusOneAndPlusOne}   
\end{align}
where the pre-factor $C_{q_{1}}$ depends only on $q_{1}$.
\vspace*{0.01in}\\
(ii) Let $\bm{T}=(\theta_{ij})_{i,j=1}^{m,d}$ have independent standard Gaussian entries.
Then the Hausdorff distance \eqref{HausdorffbetnNormBallsComposedWithLinear} satisfies
\begin{align}
\mathbb{E}\:\delta\leq C \left(m^{\frac{1}{q_{2}}-\frac{1}{q_{1}}}-1\right) 2^{5/q_1}\left(\log m\right)^{1/q_{1}}\left(\gamma_{2} +\gamma_{q_{1}}\:\mathbb{E}\:\underset{i,j}{\max}\:\lvert\theta_{ij}\rvert\right) + 2^{1/q_{1}}\gamma_{q_{1}}
\label{ExpectedHausdorffAijStdNormal}    
\end{align}
where $C>0$ is a constant, and $\gamma_r := \left(\mathbb{E} \lvert X \rvert ^{r}\right)^{1/r}$, $r\geq 1$, is the $L_{r}$ norm of a standard Gaussian random variable $X$. In particular, $\gamma_r \asymp \sqrt{r}$, i.e., there exist positive constants $c_1,c_2$ such that $c_1\sqrt{r}\leq \gamma_r \leq c_2 \sqrt{r}$ for all $r\geq 1$.
\end{proposition}
\begin{proof}
(i) Following \cite[Thm. 1]{bennett1975norms}, we bound the expected operator norm as $\mathbb{E}\|\bm{T}\|_{2\rightarrow q_{1}} \leq C_{q_{1}} \max\{m^{1/q_{1}},\sqrt{d}\}$. Combining this with \eqref{UpperBoundLinearComposition}, the result follows.\\
(ii) The expected $2\rightarrow q_{1}$ operator norm bound, in this case, follows from specializing more general bound\footnote{The operator norm bound in \cite[Thm. 1.1]{guedon2017expectation} is more general on two counts. First, the operator norm considered there is $p^{*}\rightarrow q$ where $1 \leq p^{*} \leq 2 \leq q \leq \infty$. Second, the result therein allows nonuniform deterministic scaling of the standard Gaussian entries of $\bm{T}$.} given in \cite[Thm. 1.1]{guedon2017expectation}. Specifically, we get
\begin{align}
\mathbb{E}\|\bm{T}\|_{2\rightarrow q_{1}} \leq C 2^{5/q_{1}}\left(\log m\right)^{1/q_{1}}\left(\gamma_2 +\gamma_{q_{1}}\:\mathbb{E}\:\underset{i,j}{\max}\:\lvert a_{ij}\rvert \right) + 2^{1/q_{1}}\gamma_{q_{1}},
\label{2toqOpNormBoundStdGaussian}\end{align}
where $C,\gamma_2,\gamma_{q_{1}}$ are as in the statement. Combining \eqref{2toqOpNormBoundStdGaussian} with \eqref{UpperBoundLinearComposition}, we obtain \eqref{ExpectedHausdorffAijStdNormal}.
\end{proof}
\section{Integral Version and Application}\label{sec:integral}
We now consider a further generalization of \eqref{HausdorffbetnNormBallsComposedWithLinear} given by
\begin{align}
\delta\left(\mathcal{K}_{1},\mathcal{K}_{2}\right) = \underset{\|\bm{y}\|_{2} = 1}{\sup}\quad \int_{0}^{t}\left(\|\bm{T}(\tau)\bm{y}\|_{q_{2}}- \|\bm{T}(\tau)\bm{y}\|_{q_{1}}\right)\:\differential\tau,  \qquad 1 \leq q_{2} < q_{1} \leq \infty,  
\label{HausdorffIntegral}    
\end{align}
where for each $\tau\in[0,t]$, the matrix $\bm{T}(\tau)\in\mathbb{R}^{m\times d}$, $m\leq d$, is smooth in $\tau$ and has full row rank $m$. 

As before, let $p_{1},p_{2}$ denote the H\"{o}lder conjugates of $q_{1},q_{2}$, respectively. We can interpret \eqref{HausdorffIntegral} as computing the Hausdorff distance between two compact convex sets in $\mathbb{R}^{d}$ obtained by first taking linear transformations of the $m$-dimensional $p_1$ and $p_2$ unit norm balls via $\bm{T}^{\top}(\tau)\in\mathbb{R}^{d\times m}$ for fixed $\tau\in[0,t]$, and then taking respective Minkowski sums for varying $\tau$ and finally passing to the limit. In particular, if we let $\mathcal{P}_{i}:=\{\bm{v}\in\mathbb{R}^{m}\mid\|\bm{v}\|_{p_{i}}\leq 1\}$ for $i\in\{1,2\}$, then \eqref{HausdorffIntegral} computes the Hausdorff distance between the $d$ dimensional compact convex sets
\begin{subequations}
\begin{align}
&\mathcal{K}_1\equiv\int_{0}^{t}\bm{T}^{\top}(\tau)\mathcal{P}_{1}\differential\tau := \underset{\Delta\downarrow 0}{\lim}\sum_{i=0}^{\lfloor t/\Delta \rfloor}\Delta \bm{T}^{\top}(i\Delta)\mathcal{P}_{1},\\
&\mathcal{K}_2\equiv\int_{0}^{t}\bm{T}^{\top}(\tau)\mathcal{P}_{2}\differential\tau := \underset{\Delta\downarrow 0}{\lim}\sum_{i=0}^{\lfloor t/\Delta \rfloor}\Delta \bm{T}^{\top}(i\Delta)\mathcal{P}_{2},
\end{align}
\label{SetValuedIntegrals}
\end{subequations}
i.e., the sets under consideration are set-valued Aumann integrals \cite{aumann1965integrals} and the symbol $\sum$ denotes the Minkowski sum. That the sets in \eqref{SetValuedIntegrals} are convex is a consequence of the Lyapunov convexity theorem \cite{liapounoff1940fonctions,halmos1948range}.

Notice that in this case, \eqref{UpperBoundLinearComposition} directly yields
\begin{align}
\delta\left(\mathcal{K}_{1},\mathcal{K}_{2}\right) \leq \int_{0}^{t}\underset{\|\bm{y}\|_{2} = 1}{\sup}&\left(\|\bm{T}(\tau)\bm{y}\|_{q_{2}}- \|\bm{T}(\tau)\bm{y}\|_{q_{1}}\right)\:\differential\tau \nonumber \\
&\quad \quad \quad\quad \quad\leq \left(m^{\frac{1}{q_{2}}-\frac{1}{q_{1}}}-1\right) \int_{0}^{t}\|\bm{T}(\tau)\|_{2\rightarrow q_{1}}\:\differential\tau.
\label{IntegralBound}    
\end{align}
A different way to deduce \eqref{IntegralBound} is to utilize the definitions \eqref{SetValuedIntegrals}, and then combine the Hausdorff distance property in \cite[Lemma 2.2(ii)]{de1976differentiability} with a limiting argument. This gives
\begin{align}
\delta\left(\int_{0}^{t}\bm{T}^{\top}(\tau)\mathcal{P}_{1}\differential\tau, \int_{0}^{t}\bm{T}^{\top}(\tau)\mathcal{P}_{2}\differential\tau\right) \leq \int_{0}^{t}\delta\left(\bm{T}^{\top}(\tau)\mathcal{P}_{1},\bm{T}^{\top}(\tau)\mathcal{P}_{2}\right)\differential\tau.    
\label{deltaineqIntegral}    
\end{align}
For a fixed $\tau\in[0,t]$, the integrand in the right hand side of \eqref{deltaineqIntegral} is precisely \eqref{HausdorffbetnNormBallsComposedWithLinear}, hence using Proposition \ref{prop:upperbound} we again arrive at \eqref{IntegralBound}.

As a motivating application, consider two controlled linear dynamical agents with identical dynamics given by the ordinary differential equation
\begin{align}
\dot{\bm{x}}^{i}(t) =\bm{A}(t)\bm{x}^{i}(t) + \bm{B}(t)\bm{u}^{i}(t), \quad i\in\{1,2\},
\label{LTVdyn}    
\end{align}
where $\bm{x}^{i}(t)\in\mathbb{R}^{d}$ is the state and $\bm{u}^{i}(t)\in\mathbb{R}^{m}$ is the control input for the $i$th agent at time $t$. Suppose that the system matrices $\bm{A}(t),\bm{B}(t)$ are smooth measurable functions of $t$, and that the initial conditions for the two agents have the same compact convex set valued uncertainty, i.e., $\bm{x}^{i}(t=0)\in$ compact convex$\:\mathcal{X}_0 \subset \mathbb{R}^{d}$. Furthermore, suppose that the input uncertainty sets for the two systems are given by different unit norm balls
\begin{align}
\mathcal{U}^i := \{\bm{u}^{i}(\tau)\in\mathbb{R}^{m}\mid \|\bm{u}^{i}(\tau)\|_{p_{i}}\leq 1\:\text{for all}\:\tau\in[0,t]\}, \quad i\in\{1,2\},
\label{inputnormballs}    
\end{align}
\noindent such that $1 \leq p_1 < p_2 \leq \infty$. Given these set-valued uncertainties, the ``reach sets" $\mathcal{X}_{t}^{i}$, $i\in\{1,2\}$, are defined as the respective set of states each agent may reach at a given time $t>0$. Specifically, for $i\in\{1,2\}$, and $\mathcal{U}^i$ given by \eqref{inputnormballs}, the reach sets are
\begin{align}
\!\mathcal{X}_{t}^{i}&:=\!\!\!\bigcup_{\text{measurable}\;\bm{u}^{i}(\cdot)\in\mathcal{U}^{i}}\!\!\!\big\{\bm{x}^{i}(t)\in\mathbb{R}^{d} \mid \dot{\bm{x}}^{i}(t) =\bm{A}(t)\bm{x}^{i}(t) + \bm{B}(t)\bm{u}^{i}(t), \quad i\in\{1,2\},\nonumber \\ &\quad \quad  \bm{x}^{i}(t=0)\in\text{compact convex}\:\mathcal{X}_0,\quad \bm{u}^{i}(\tau)\in\mathcal{U}^{i}\;\text{for all}\; 0\leq \tau\leq t\big\}.
\label{DefReachSets}    
\end{align}
As such, there exists a vast literature \cite{pecsvaradi1971reachable,witsenhausen1972remark,chutinan1999verification,kurzhanski1997ellipsoidal,varaiya2000reach,le2010reachability,althoff2021set,haddad2023curious,haddad2021anytime} on reach sets and their numerical approximations. In practice, these sets are of interest because their separation or intersection often imply safety or the lack of it. It is natural to quantify the distance between reach sets or their approximations in terms of the Hausdorff distance \cite{guseinov2007approximation,dueri2016consistently,halder2020smallest}, and in our context, this amounts to estimating $\delta\left(\mathcal{X}_{t}^{1},\mathcal{X}_{t}^{2}\right)$. 

Since $1 \leq p_1 < p_2 \leq \infty$, we have the norm ball inclusion $\mathcal{U}^{1} \subset \mathcal{U}^{2}$, and consequently $\mathcal{X}_{t}^{1} \subset \mathcal{X}_{t}^{2}$. We next show that $\delta\left(\mathcal{X}_{t}^{1},\mathcal{X}_{t}^{2}\right)$ is exactly of the form \eqref{HausdorffIntegral}.
\begin{theorem}\label{thm:HausdorffBetnLTVDiffNormBallInput}
(\textbf{Hausdorff distance between linear systems' reach sets with norm ball input uncertainty})
Consider the reach sets \eqref{DefReachSets} with input set valued uncertainty \eqref{inputnormballs}. For $\tau\leq t$, let $\bm{\Phi}(t,\tau)$ be the state transition matrix (see e.g., \cite[Ch. 1.3]{brockett1970finite}) associated with \eqref{LTVdyn}. Denote the H\"{o}lder conjugate of $p_1$ as $q_1$, and that of $p_{2}$ as $q_{2}$, i.e., $1/p_1 + 1/q_1 = 1$ and $1/p_2 + 1/q_2 = 1$. Then $1\leq q_2 < q_1 \leq \infty$, and the Hausdorff distance
\begin{align}
\delta\left(\mathcal{X}_{t}^{1},\mathcal{X}_{t}^{2}\right) = \underset{\|\bm{y}\|_{2} = 1}{\sup}\int_{0}^{t}\left(\|\left(\bm{\Phi}(t,\tau)\bm{B}(\tau)\right)^{\top}\bm{y}\|_{q_{2}}- \|\left(\bm{\Phi}(t,\tau)\bm{B}(\tau)\right)^{\top}\bm{y}\|_{q_{1}}\right)\differential\tau.
\label{HausdorffLTV}    
\end{align}
\end{theorem}
\begin{proof}
We have
\begin{align}
\mathcal{X}^{i}_{t} = \bm{\Phi}(t,0)\mathcal{X}_{0} \:\dotplus\int_{0}^{t}\bm{\Phi}(t,\tau)\bm{B}(\tau)\mathcal{U}^{i}\:\differential\tau, \quad i\in\{1,2\},  
\label{LTVreachsetAsIntegral}    
\end{align}
where $\dotplus$ denotes the Minkowski sum and the second summand in \eqref{LTVreachsetAsIntegral} is a set-valued Aumann integral.

Since the support function is distributive over the Minkowski sum, following \cite[Prop. 1]{haddad2020convex} and \eqref{DefSptFn}, from \eqref{LTVreachsetAsIntegral} we find that
\begin{align}
h_{i}\left(\bm{y}\right):=h_{\mathcal{X}^{i}_{t}}\left(\bm{y}\right) = \left(\underset{\bm{x}_{0}\in\mathcal{X}_{0}}{\sup}\!\langle\bm{y},\bm{\Phi}(t,0)\bm{x}_0\rangle\right)\! +\!\! \int_{0}^{t}\!\!\underset{\bm{u}^{i}(\tau)\in\mathcal{U}^{i}}{\sup}\!\langle\bm{y},\bm{\Phi}(t,\tau)\bm{B}(\tau)\bm{u}^{i}(\tau)\rangle\:\differential\tau,
\label{hi}
\end{align}
wherein $i\in\{1,2\}$ and the sets $\mathcal{U}^{i}$ are given by \eqref{inputnormballs}. Next, we follow the same arguments as in \cite[Thm. 1]{haddad2022certifying} to simplify \eqref{hi} as
\begin{align}
h_{i}\left(\bm{y}\right) = \left(\underset{\bm{x}_{0}\in\mathcal{X}_{0}}{\sup}\langle\bm{y},\bm{\Phi}(t,0)\bm{x}_0\rangle\right) + \int_{0}^{t}\|\left(\bm{\Phi}(t,\tau)\bm{B}(\tau)\right)^{\top}\bm{y}\|_{q_{i}}\:\differential\tau,    
\label{Anotherhi}    
\end{align}
where $q_{i}$ is the H\"{o}lder conjugate of $p_{i}$. 
Then \eqref{HausdorffSptFn} together with \eqref{Anotherhi} yield \eqref{HausdorffLTV}. 
\end{proof}

\begin{corollary}\label{corr:UpperBoundLTVHausdorff}
Using the same notations of Theorem \ref{thm:HausdorffBetnLTVDiffNormBallInput}, we have
\begin{align}
\delta\left(\mathcal{X}_{t}^{1},\mathcal{X}_{t}^{2}\right) \leq \left(m^{\frac{1}{q_{2}}-\frac{1}{q_{1}}}-1\right)\int_{0}^{t}\|\bm{\Phi}(t,\tau)\bm{B}(\tau)\|_{p_1\rightarrow 2}\:\differential\tau.
\label{LTVupperbndHausdorff}    
\end{align}
\end{corollary}
\begin{proof}
From \eqref{IntegralBound}, we obtain the estimate
\begin{align}
\delta\left(\mathcal{X}_{t}^{1},\mathcal{X}_{t}^{2}\right) \leq \left(m^{\frac{1}{q_{2}}-\frac{1}{q_{1}}}-1\right)\int_{0}^{t}\|\left(\bm{\Phi}(t,\tau)\bm{B}(\tau)\right)^{\top}\|_{2\rightarrow q_{1}}\:\differential\tau.
\label{estimate}
\end{align}
Recall that the norm of a linear operator is related to the norm of its adjoint via
$$\|\left(\cdot\right)^{\top}\|_{\alpha\rightarrow\beta} = \|\cdot\|_{\beta^{*}\rightarrow\alpha^{*}},$$
where $\alpha^{*},\beta^{*}$ are the H\"{o}lder conjugates of $\alpha,\beta$, respectively. Using this fact in \eqref{estimate} completes the proof.
\end{proof}
\begin{remark}\label{remark:LTI}
In the special case of a linear time invariant dynamics, the matrices $\bm{A},\bm{B}$ in \eqref{LTVdyn} are constants and $\bm{\Phi}(t,\tau)=\exp((t-\tau)\bm{A})$. In that case, Theorem \ref{thm:HausdorffBetnLTVDiffNormBallInput} and Corollary \ref{corr:UpperBoundLTVHausdorff} apply with these additional simplifications. 
\end{remark}
\begin{remark}\label{remark:Longtime}
As $t$ increases, we expect the bound \eqref{IntegralBound} to become more conservative. Likewise, the gap between \eqref{HausdorffLTV} and \eqref{LTVupperbndHausdorff} is expected to increase with $t$.
\end{remark}

\vspace*{0.1in}

\noindent\textbf{Example.} Consider the linearized equation of motion of a satellite \cite[p. 14-15]{brockett1970finite} of the form \eqref{LTVdyn} with four states, two control inputs, and constant system matrices
\begin{align}
\bm{A}(t)\equiv \begin{bmatrix}
0 & 1 & 0 & 0\\
3\omega^2 & 0 & 0 & 2\omega\\
0 & 0 & 0 & 1\\
0 & -2\omega & 0 & 0
\end{bmatrix}, \quad \bm{B}(t)\equiv \begin{bmatrix}
0 & 0\\
1 & 0\\
0 & 0\\
0 & 1
\end{bmatrix},
\label{LinEx}
\end{align}
for some fixed parameter $\omega$. The input components denote the radial and tangential thrusts, respectively. We consider two cases: the inputs have set-valued uncertainty of the form \eqref{inputnormballs} with $p_1 = 2$ (unit Euclidean norm-bounded thrust) and with $p_2 = \infty$ (unit box-valued thrust). We have \cite[p. 41]{brockett1970finite}
$$\bm{\Phi}(t,\tau)\bm{B} = \begin{bmatrix}
\dfrac{\sin(\omega(t-\tau))}{\omega} & \dfrac{2(1-\cos(\omega(t-\tau)))}{\omega}\\
& \\
\cos(\omega(t-\tau)) & 2\sin(\omega(t-\tau))\\
& \\
-\dfrac{2(1-\cos(\omega(t-\tau)))}{\omega} & \dfrac{-3\omega(t-\tau) + 4 \sin(\omega(t-\tau))}{\omega}\\
& \\
-2\sin(\omega(t-\tau)) & -3 + 4\sin(\omega(t-\tau))
\end{bmatrix},$$
 for $0\leq \tau < t$, and the integrand in the right hand side of \eqref{LTVupperbndHausdorff} equals to the maximum singular value of the above matrix. For $\omega = 3$ and $t\in[0,2]$, Fig. \ref{fig:satellite} shows the time evolution of the numerically estimated Hausdorff distance \eqref{HausdorffLTV} and the upper bound \eqref{LTVupperbndHausdorff} between the reach sets given by \eqref{DefReachSets} with the same compact convex initial set $\mathcal{X}_0\subset\mathbb{R}^{4}$, i.e., between $\mathcal{X}_t^{1}$ and $\mathcal{X}_t^{2}$ resulting from the unit $p_1=2$ and $p_2=\infty$ norm ball input sets, respectively.

\begin{figure}%
    \centering
    \includegraphics[width=0.85\linewidth]{ 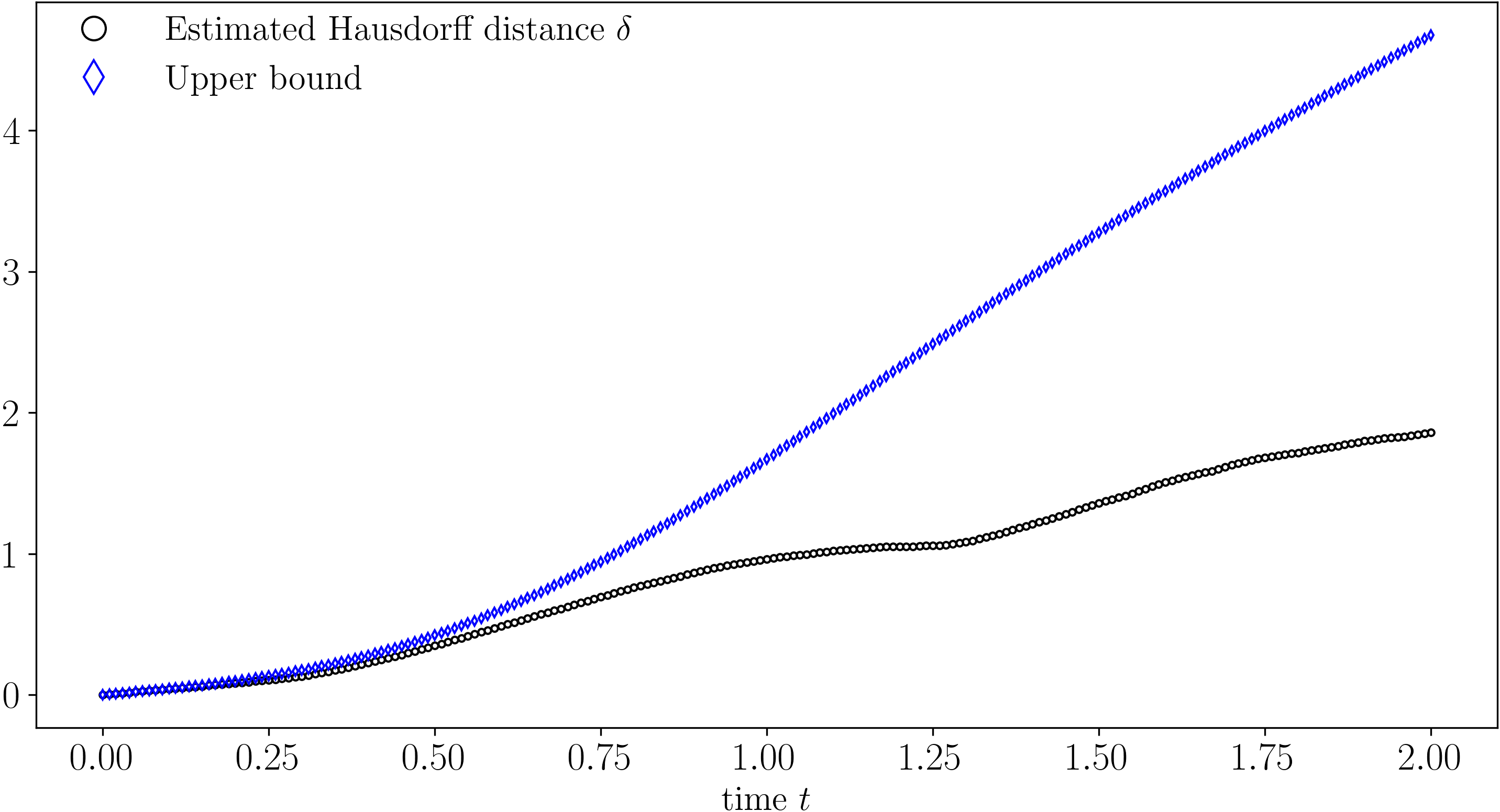}
    \caption{The numerically estimated Hausdorff distance \eqref{HausdorffLTV} and the upper bound \eqref{LTVupperbndHausdorff} for the four state, two input linear system given in \eqref{LinEx}.}%
    \label{fig:satellite}%
\end{figure} 

\section{Conclusions}\label{sec:conclusion}
In this work, we studied the Hausdorff distance between two different norm balls in an Euclidean space and derived closed-form formula for the same. In $d$ dimensions, we provide results for the $\ell_p$ norm balls parameterized by $p$ where $1\leq p \leq \infty$, as well as for the polyhedral $D$-norm balls parameterized by $k$ where $1\leq k \leq d$. We then investigated a more general setting: the Hausdorff distance between two convex sets obtained by transforming two different $\ell_p$ norm balls via a given linear map. In this setting, while we do not know a general closed-form formula for an arbitrary linear map, we provide upper bounds for the Hausdorff distance or its expected value depending on whether the linear map is arbitrary or random. Our results make connections with the literature on hypercontractive operator norms, and on the norms of random linear maps. We then focus on a further generalization: the Hausdorff distance between two set-valued integrals obtained by applying a parametric family of linear maps to the unit $\ell_p$ norm balls, and then
taking the Minkowski sums of the resulting sets in a limiting sense. As an illustrative application, we show that the problem of computing the Hausdorff distance between the reach sets of a linear time-varying dynamical system with different unit $\ell_p$ norm ball-valued input uncertainties, leads to this set-valued integral setting.

We envisage several directions of future work. It is natural to further explore the qualitative properties of the nonconvex landscape \eqref{HausdorffbetnNormBallsComposedWithLinear}, and to design efficient algorithms in computing the Hausdorff distance for the same. It should also be of interest to pursue a systems-control theoretic interpretation of \eqref{HausdorffLTV} and \eqref{LTVupperbndHausdorff} in terms of functionals of the associated controllability Gramian.

\backmatter

\bmhead{Supplementary information}
`Not applicable'.

\bmhead{Availability of data and materials}
The datasets that support our results in this study are available from the authors on request.
\bibliography{references.bib}


\end{document}